\documentclass[reqno,12pt]{amsart}
\usepackage{amsmath,amsfonts,amssymb,amsmath,enumerate,setspace,graphics,graphicx,latexsym,amscd,subfigure,hyperref,comment,hyperref}

\usepackage{fullpage}
\usepackage[usenames]{xcolor}
\usepackage{verbatim}

\newtheorem{thm}{Theorem}[section]
\newtheorem{cor}[thm]{Corollary}
\newtheorem{lem}[thm]{Lemma}
\newtheorem{lemma}[thm]{Lemma}
\newtheorem{prop}[thm]{Proposition}
\newtheorem{defn}[thm]{Definition}
\newtheorem{definition}[thm]{Definition}
\newtheorem{remark}[thm]{Remark}
\newtheorem{ex}[thm]{Example}

\definecolor{dgreen}{RGB}{0,128,0}

\newcommand{\orb}{\mathcal O}
\newcommand{\lc}{\ell}
\newcommand{\D}{\mathcal D}
\newcommand{\R}{\mathbf R}
\newcommand{\Z}{\mathbf Z}
\newcommand{\N}{\mathbf N}
\newcommand{\euc}{\operatorname{Euc}}
\newcommand{\can}{\operatorname{can}}
\newcommand{\geuc}{g_{\euc}}
\newcommand{\bs}{\backslash}
\newcommand{\Stek}{\operatorname{Stek}}

\newcommand{\Om}{\Omega}
\newcommand{\ab}{\overline{A}}
\newcommand{\cb}{\overline{L}}

\newcommand{\orbt}{\widetilde{\orb}}
\newcommand{\conf}{\mathcal C}
\newcommand{\F}{\mathcal F}
\newcommand{\Ut}{\widetilde{U}}
\newcommand{\At}{\widetilde{A}}

\newcommand{\Lc}{\mathcal{L}}
\newcommand{\ric}{{\rm Ric}}
\newcommand{\vol}{{\rm vol}}

\newcommand{\iso}{\mathcal I}

\newcommand{\OP}{\mathcal P}
\newcommand{\inj}{{\rm inj}}

\newcommand{\hyp}{\mathbf H}

\newcommand{\vp}{\varphi}
\newcommand{\cc}{(\widetilde{U},\Gamma_U, \varphi_U)}
\newcommand{\wtu}{\widetilde{U}}
\newcommand{\C}{\mathbf C}
\newcommand{\ohd}{\orb_{\rm HD}}
\newcommand{\pv}{\mathbf p}
\newcommand{\av}{\mathbf a}
\newcommand{\gqp}{\Gamma_{q,\pv}}

\begin{document}

\title[]{Spectral geometry of the Steklov problem on orbifolds}

\author{Teresa Arias-Marco}
\address{Department of Mathematics, University of Extremadura,  Av. de Elvas s/n, 06006 Badajoz, Spain}
\email{ariasmarco@unex.es}
\author{Emily B. Dryden}
\address{Department of Mathematics, Bucknell University, Lewisburg, PA 17837, USA}
\email{emily.dryden@bucknell.edu}
\author{Carolyn S. Gordon}
\address{Department of Mathematics, 6188 Kemeny, Dartmouth College, Hanover, NH  03755, USA}
\email{csgordon@dartmouth.edu}
\author{Asma Hassannezhad}
\address{Max-Planck Institute for Mathematics, Vivatsgasse 7, 53111 Bonn, Germany}
\email{hassannezhad@mpim-bonn.mpg.de}
\author{Allie Ray}
\address{Trinity College, Department of Mathematics, 300 Summit St., Hartford, CT 06106, USA}
\email{allie.ray@trincoll.edu}
\author{Elizabeth Stanhope}
\address{Department of Mathematical Sciences, Lewis \& Clark College, Portland, OR 97219, USA}
\email{stanhope@lclark.edu}
\email{}

\date{}

\begin{abstract}

We consider how the geometry and topology of a compact $n$-dimensional Riemannian orbifold with boundary relates to its Steklov spectrum. In two dimensions, motivated by work of A. Girouard, L. Parnovski, I. Polterovich and D. Sher in the manifold setting, we compute the precise asymptotics of the Steklov spectrum in terms of only boundary data.  As a consequence, we prove that the Steklov spectrum detects the presence and number of orbifold singularities on the boundary of an orbisurface and it detects the number each of smooth and singular boundary components.  Moreover, we find that the Steklov spectrum also determines the lengths of the boundary components modulo an equivalence relation, and we show by examples that this result is the best possible.   We construct various examples of Steklov isospectral Riemannian orbifolds which demonstrate that these two-dimensional results do not extend to higher dimensions.  

In dimension two, we show that a flat disk is not only Steklov isospectral to a cone but, in fact, a disk and cone of appropriate size have identical Dirichlet-to-Neumann operators. This provides a counterexample to the inverse tomography problem in the orbifold setting and contrasts with results of Lassas and Uhlmann in the manifold setting.
 
In another direction, we obtain upper bounds on the Steklov eigenvalues of a Riemannian orbifold in terms of the isoperimetric ratio and a conformal invariant. We generalize results of B. Colbois, A. El Soufi and A. Girouard, and the fourth author to the orbifold setting; in the process, we gain a sharpness result on these bounds that was not evident in the  manifold setting.  In dimension two, our eigenvalue bounds are solely in terms of the orbifold Euler characteristic and the number each of smooth and singular boundary components.
\end{abstract}

\maketitle

%%%%%%%%%%%%%%%%%%%%%%%%%%%%%%%%%%%%%%%%%%%%%%%%%%%%%%%%%%%%%%%%%%%%%%%%%%%%%%%%%%%%%%%%%%%%%%%%%%%%%%%%%%%%%%%%%%%%%%%%%%%%%%%%%%%%%%%%%%%%%%%%%%%%%%

\section{Introduction}
Let $(M,g)$ be a compact $n$-dimensional Riemannian manifold with smooth boundary.  For $f\in C^\infty(M)$, denote by $ \partial_\nu f$ the outward normal derivative of $f$ along $\partial M$.  The \emph{Dirichlet-to-Neumann operator}  $\mathcal{D}_{(M,g)}:C^\infty(\partial M)\to C^\infty(\partial M)$ sends a function $u\in C^\infty(\partial M)$ to $\partial_\nu (Hu)$, where $Hu$ is the unique harmonic extension of $u$ to $M$.  This operator is closely related to the so-called voltage-to-current operator that arises in electrical impedance tomography.   The spectrum of this operator is called the \emph{Steklov spectrum}.  Equivalently, the Steklov spectrum {$\Stek(M,g)$} consists of those $\sigma\in\R$ for which there exists a nonzero solution $f \in C^{\infty}({M})$ of
\[ \begin{cases}
\Delta f = 0 & {\rm in} \ {M}, \\
      \partial_\nu f = \sigma f & {\rm on} \ \partial  {M}.
         \end{cases}
\]
where $\Delta$ is the Laplace-Beltrami operator of $M$.
The Steklov spectrum $0=\sigma_1\le\sigma_2\le\cdots\nearrow\infty$ is discrete, non-negative and {unbounded,} and each eigenvalue has finite multiplicity.  
   
The Steklov problem on compact Riemannian manifolds with boundary was introduced by V. A. Steklov in 1902 \cite{Stek02} (see also \cite{K.et.al} for a historical discussion) and has recently seen a surge of interest from the spectral geometry community.  We recommend the excellent expository article \cite{GP}, which surveys results on eigenvalue asymptotics, questions of isospectrality and rigidity, and geometric bounds on eigenvalues, as well as other results.   

In this article we initiate a study of the Steklov problem on compact $n$-dimensional Riemannian orbifolds with boundary. A Riemannian orbifold is a mildly singular generalization of a Riemannian manifold.  Originally introduced by I. Satake \cite{Satake56} in 1956, orbifolds were later popularized by W. Thurston \cite{Th}.  
Today orbifolds are approached from a variety of viewpoints.  A sampling of these perspectives may be found, for instance, in \cite{ALR, CHK, GH, PeterScott}.

 We seek to understand how the geometry and topology of a Riemannian orbifold relates to its Steklov spectrum.  Our goals are not only to extend known results to the orbifold setting, but also to compare the Steklov spectra of orbifolds with those of manifolds.  In particular, one of our motivating questions is the extent to which the Steklov spectrum is affected by, and detects the presence of, orbifold singularities.

 We focus on the following:
\begin{enumerate}
\item[(i)] For arbitrary two-dimensional compact Riemannian orbifolds (also called Riemannian orbisurfaces) with boundary, we compute the precise asymptotics of the Steklov spectrum in terms of only boundary data.  As a consequence, we show that the asymptotics of the Steklov spectrum fully determine the \emph{topology} of the boundary; in particular, the asymptotics detect the presence and number of orbifold singularities on the boundary.   We also determine the extent to which the {asymptotics of the Steklov spectrum} encode the \emph{geometry} of the boundary. This work is motivated by and builds on the beautiful paper of A. Girouard, L. Parnovski, I. Polterovich and D. Sher  \cite{GPPS}.   Through the construction of Steklov isospectral Riemannian orbifolds, we show that these two-dimensional results do not extend to higher dimensions.

\item[(ii)] Also in dimension two, we show that not only the Steklov spectrum, but the Dirichlet-to-Neumann operator itself does \emph{not} detect the presence of interior singularities and does not determine the orbifold Euler characteristic.  In particular, while the Euclidean disk in $\R^2$ is known to be uniquely determined by its Steklov spectrum within the class of Euclidean domains, we find that it has the \emph{same Dirichlet-to-Neumann operator} as a cone of appropriate size.  Our example contrasts with known inverse results (e.g., \cite{LassasUhlmann}) for the Dirichlet-to-Neumann operator on smooth manifolds.  The question of whether the Dirichlet-to-Neumann operator on the boundary of a compact Riemannian manifold $(M,g)$ uniquely determines the Riemannian manifold up to isometry is sometimes referred to as the inverse tomography problem and may be interpreted as determining the electrical conductivity of a medium via voltage and current measurements on the boundary.  (See \cite{LU}.)

\item[(iii)] We adapt to the orbifold setting results of  {B. Colbois, A. El Soufi and A. Girouard  \cite{CEG}} and the fourth author {\cite{H}} concerning upper bounds on the Steklov eigenvalues in terms of the isoperimetric ratio and a conformal invariant.  By extending the bounds to the orbifold setting, we are able to gain some information that was not evident in the  manifold setting concerning the sharpness of the bounds.  In dimension two, we also obtain eigenvalue bounds in terms of the orbifold Euler characteristic and the number each of smooth and {singular} boundary components. 
\end{enumerate}

Before describing our results in more detail, we note the contrast between (i), which shows that the Steklov spectrum contains considerable information about the boundary of an orbisurface, and the negative results (ii).  This contrast is consistent with the fact that the Dirichlet-to-Neumann operator is a pseudodifferential operator whose full symbol depends on the Riemannian metric only on an arbitrarily small neighborhood of the boundary (see \cite{LU}).  As a consequence, the \emph{asymptotic behavior} of the eigenvalues can reveal geometric information about the boundary but not the interior.  Little is known, even in the {manifold} setting, about the extent to which the \emph{full Steklov spectrum} may encode information about the interior geometry.  In two dimensions, moreover, the Dirichlet-to-Neumann operator and hence also the Steklov spectrum are not affected by conformal changes of metric away from the boundary.

In the following subsections, we elaborate on each of (i)-(iii).

\subsection{Asymptotics of the Steklov spectrum on orbisurfaces}\label{asymp}

In \cite{GPPS}, A. Girouard, L. Parnovski, I. Polterovich and D. Sher studied precise asymptotics of the Steklov spectrum for  Riemannian surfaces $(M,g)$  (with smooth boundary).   In this setting, the topology and geometry of the boundary of $M$ are completely expressed by the number $r$ of boundary components and their lengths $\ell_1, \dots, \ell_r$.   The main results of \cite{GPPS} are an explicit computation of the asymptotics of the Steklov spectrum in terms of this boundary data and, conversely, a proof that the Steklov spectrum completely determines the number and lengths of the boundary components.  To obtain the asymptotics, a key step is to show that $\Stek(M,g)$ is asymptotic to the Steklov spectrum of a disjoint union of flat disks:
$$\Stek(M,g)\sim \Stek(D(\ell_1)\sqcup\cdots\sqcup D(\ell_r))$$
where $D(\ell)$ is the flat disk of circumference $\ell$, and where for sequences $A=\{a_j\}$ and $B=\{b_j\}$, we write $A\sim B$ to mean $a_j-b_j=O(j^{-\infty})$. The Steklov eigenfunctions and Steklov spectrum of a flat disk $D(\ell)$ are easily computed:  the eigenfunctions are the restrictions of the homogeneous harmonic polynomials on $\R^2$ to the disk, and the Steklov spectrum  is given by the multiset $\{0\}\sqcup \frac{2\pi}{\ell}\N\sqcup \frac{2\pi}{\ell}\N.$  See Example~\ref{spec.ball} for details.  (Note:  As is common in the literature, we are using the term \emph{Steklov eigenfunction} to mean the harmonic extensions to the disk of the Dirichlet-to-Neumann eigenfunctions on the boundary.)

In contrast to the class of smooth compact surfaces, where the only topological invariant of the boundary is the number of boundary components, the boundaries of 
compact orbisurfaces are expressed by two topological invariants.  Each boundary component of a compact orbisurface $\orb$ is either a circle or the quotient of a circle by a reflection.  We will call these type I and type II boundary components, respectively.  The boundary components of type II are closed orbifolds with two singular points.  Thus the topology of the boundary is given by the numbers $r$ and $s$ of boundary components of types I and II, respectively.  The geometry is expressed by the {boundary component} lengths ${\ell}_1,\dots, {\ell}_r$ and ${\bar{\ell_1},\dots, \bar{\ell_s}}$ corresponding to types I and II, respectively.

{Let $\ohd(\ell)$  be a half disk obtained as the quotient of a flat disk of circumference $2\ell$ by a reflection; we will view $\ohd(\ell)$ as our model Riemannian orbisurface with a single type II boundary component of boundary length $\ell$.  } The Steklov eigenfunctions of this orbisurface are precisely the functions that lift to reflection invariant Steklov eigenfunctions on the disk.  Consequently we find that $\Stek(\ohd(\ell))=\{0\}\sqcup \frac{\pi}{\ell}\N$ (see Example~\ref{half_disk_orb}).   Note the contrast with the spectrum of a disk: here each eigenvalue is simple. 
%\newr{\cite{LPP06}}

In the same vein as the main results of A. Girouard, L. Parnovski, I. Polterovich and D. Sher, we prove:

\begin{thm}\label{h-thm.asympt} Let $(\orb,g)$ be a {compact Riemannian orbisurface} with boundary consisting of $r$ type I boundary components of lengths ${\ell}_1,\dots, {\ell}_r$ and $s$ type II boundary components of lengths ${\bar{\ell_1},\dots, \bar{\ell_s}}$. Then
$$\Stek(\orb,g)\sim \Stek(S({\ell}_1,\dots, {\ell}_r;{\bar{\ell_1},\dots, \bar{\ell_s}})),$$
where $$S({\ell}_1,\dots, {\ell}_r;\bar{\ell_1},\dots, \bar{\ell_s}) =D(\ell_1)\sqcup\cdots\sqcup D(\ell_r)\sqcup \ohd(\bar{\ell_1})\sqcup\dots\sqcup\ohd(\bar{\ell_s}).$$
\end{thm}

\begin{thm}\label{thm-boundary.top} The Steklov spectrum of a compact orbisurface with boundary determines the number of boundary components of each type.  It also determines the lengths of the boundary components modulo an equivalence relation.   In particular, the spectrum completely determines the topology of the boundary, including the number of orbifold singularities.

\end{thm}

{In contrast to the smooth case studied in \cite{GPPS}, the Steklov spectrum does not always fully encode the lengths of the boundary components, as the following counterexample illustrates.    The equivalence relation referred to in Theorem~\ref{thm-boundary.top} is generated by the interchanges described in this example}.

\begin{ex}\label{ell1ell2} 
Let  $\lc_1$ and $\lc_2$ be two positive numbers.   By comparing the Steklov spectra of flat disks and {the half disks obtained as their quotients by reflections}, we see immediately that the disjoint union of a flat disk $D(2\ell_1)$ of circumference $2\lc_1$ together with two copies of $\ohd(\lc_2)$ is Steklov isospectral to the disjoint union of a flat disk $D(2\ell_2)$ of circumference $2\lc_2$ together with two copies of $\ohd(\lc_1)$:  
$$\Stek\big(D(2\ell_1)\sqcup\ohd(\ell_2)\sqcup\ohd(\ell_2)\big)=\Stek\big(D(2\ell_2)\sqcup\ohd(\ell_1)\sqcup\ohd(\ell_1)\big).$$ 
Here the first orbifold has one type I boundary component of length $2\ell_1$ and two type II boundary components each of length $\ell_2$, while the second has one type I boundary component of length $2\ell_2$ and two type II boundary components each of length $\ell_1$.  However, we will see that this type of interchange is the only aspect of the boundary geometry not detected by the Steklov spectrum (see Theorem \ref{h-audiblebd}).
\end{ex}

\begin{remark}  One can give a second interpretation of Example~\ref{ell1ell2}.   A half disk $\ohd(\ell)$ may be viewed either as an orbifold, as we do in the example, or as a plane domain $\Omega$. Let $A$ and $B$ denote the semicircular and straight segments of the boundary of $\Omega$, respectively.  Thus $A$ corresponds to the boundary of $\ohd(\ell)$. A smooth function $f$ on the orbifold $\ohd(\ell)$ pulls back to a reflection-invariant smooth function on the disk.   Thus, when viewed as a function on $\Omega$, $f$ satisfies  Neumann boundary conditions on $B$.  In particular, the Steklov spectrum of the orbifold $\ohd(\ell)$ corresponds precisely to the spectrum of a mixed Neumann-Steklov problem on the plane domain.  The latter spectrum is that of the operator $C^\infty(A)\to C^\infty(A)$ which sends a function $u\in C^\infty(A)$ to $\partial_\nu (Hu)$ evaluated along $A$, where $Hu$ is the unique harmonic extension of $u$ satisfying Neumann boundary conditions on $B$.  
%(Mixed Neumann-Steklov problems are sometimes referred to as \emph{sloshing problems} as they describe the oscillations of fluid in an open container.)   
With this interpretation, Example~\ref{ell1ell2} provides an elementary example of isospectral surfaces (with multiple components) for the mixed Neumann-Steklov isospectral problem.  The construction is reminiscent of the elegantly simple isospectral constructions for mixed eigenvalue problems given in \cite{LPP06}.
 \end{remark}

\subsection{Negative inverse results}\label{isosp}
Our first negative result addresses information contained in the full Dirichlet-to-Neumann operator.
A question closely connected to electrical impedance tomography asks whether the Dirichlet-to-Neumann operator of a compact Riemannian manifold $(M,g)$ with given boundary determines $(M,g)$ uniquely (up to isometry, or in dimension two, also up to conformal change of metric away from the boundary).  Lassas and Uhlmann \cite{LassasUhlmann} answered this question affirmatively for smooth surfaces and for analytic manifolds of all dimensions.   In contrast, a very elementary construction gives a negative answer to the analogous question for orbisurfaces:

\begin{thm}\label{cone-intro} For every $k\in \N$, a flat cone of cone angle $\frac{2\pi}{k}$ is Steklov isospectral to the flat unit disk in $\R^2$ of the same radius.
In particular, the Dirichlet-to-Neumann operator does not always detect the presence of interior orbifold singularities.\end{thm}  

Perhaps the most basic inverse spectral problem for orbifolds is whether the spectrum detects the presence of orbifold singularities. Of course, Theorem~\ref{cone-intro} gives a negative answer for the Steklov spectrum.  In contrast, in the case of the Laplace spectrum the answer to this question has remained elusive, although partial results have been obtained {\cite{DGGW,GR03,Su10}}, as well as a negative answer for the Hodge Laplacian on $p$-forms {\cite{GR03}}.
The cones and disk in Theorem~\ref{cone-intro} have different areas and thus cannot be Laplace isospectral.  

One expects the Steklov spectrum to contain much less geometric information than the Laplace spectrum, at least concerning the interior of the manifold or orbifold.  However,  with the exception of Steklov isospectral surfaces obtained by conformally changing the metric away from the boundary, the examples in Theorem~\ref{cone-intro} are to our knowledge the first examples of Steklov isospectral manifolds and/or orbifolds that are not also Laplace isospectral. 

There is substantial literature addressing Laplace isospectral manifolds and orbifolds.  Recently, P. Herbrich, D. Webb and the third author \cite{GHW} showed that most of the known Laplace isospectral compact manifolds with boundary are also Steklov isospectral.  Using similar techniques, we provide examples illustrating that the results in \S\S~\ref{asymp} cannot be extended to higher dimensions. In particular,  we give examples showing  that in higher dimensions,  the Steklov spectrum does not separately determine the number of smooth and singular boundary components.  See \S \ref{sec:examples} for further comments and examples.  However, the intriguing question of whether the Steklov spectrum determines the total number of boundary components of an orbifold or manifold of arbitrary dimension remains open.

\subsection{Upper bounds}\label{bound}
Although the precise asymptotics of the Steklov spectrum fail to detect topological and geometric data of the interior of an orbifold as discussed above (see also Proposition \ref{spectrumislocal}),  upper bounds on Steklov eigenvalues inevitably depend on the global topology and geometry of the orbifold.  Upper bounds for the Steklov eigenvalues in the setting {of smooth Riemannian manifolds} have been extensively studied (see, for example, \cite{HPS75,GP10,CEG,H,CG14}).  The bounds in {\cite{CEG,H}} are in terms of geometric data such as the isoperimetric ratio and conformal invariants. Our focus here is on the extension of their results  to the orbifold setting. 

 \begin{thm}\label{h-confinv0}
 Let $(\orb,g)$ be a compact $n$-dimensional Riemannian orbifold with boundary. There exist positive constants $C_1$ and $C_2$ depending only on $n$ such that for every $k\in\N$,  
\begin{equation}\label{intro-sigmak}
\sigma_k(\orb,g) \ \vol_g(\partial\orb)^{\frac{1}{n-1}}\le \frac{C_1\conf(\orb,g)+C_2k^{\frac{2}{n}}}{\iso_g(\orb)^{1-\frac{1}{n-1}}},
\end{equation}
where $\conf(\orb,g)$ is a conformal invariant defined in Definition \ref{def:confinv} and $\iso_g(\orb)$  is the isoperimetric ratio of $\orb$ given by
\[
\iso_g(\orb)=\frac{\vol_g(\partial \orb)}{\vol_g(\orb)^\frac{n-1}{n}}.  
\]
{Moreover, i}n even dimensions $n=2m>2$, the power of $\iso_g(\orb)$ in the denominator is sharp from below; i.e., $1-\frac{1}{n-1}$ cannot be replaced by any smaller power.
\end{thm}
We prove the final statement by constructing a family of  {even dimensional orbifolds $\{(\orb_i,g_i)\}$ as quotients of  Euclidean balls such that   $\iso_{g_i}(\orb_i)\to 0$  and $\sigma_2(\orb_i,g_i)\vol_{g_i}(\partial\orb_i)^{\frac{1}{n-1}}=O(\iso_{g_i}(\orb_i)^{\frac{1}{n-1}-1})$ as $i\to\infty$. See Example \ref{h-ex} for details.}

For $n=2$, the power of $\iso_g(\orb)$ vanishes and we obtain upper bounds on Steklov eigenvalues (normalized by the length of $\partial \orb$) in terms of universal constants involving the Euler characteristic {$\chi(\orb)$} of $\orb$ {and the number of boundary components of each type}, extending results of \cite{CEG,H} to the setting of orbisurfaces.  
\begin{thm}\label{h-2dupbd0.intro}
Let $(\orb,g)$ be a compact Riemannian orbisurface with boundary consisting of $r$ type I and  $s$ type II  boundary components. Then for every $k\in\N$,
\begin{equation}\label{h-main.intro}
\sigma_k(\orb,g)\ell_g(\partial\orb)\le 
\begin{cases}
      Bk, & \text{if}\ \chi(\orb)+r+\frac{s}{2}\geq 0, \\
      -A(\chi(\orb)+r+\frac{s}{2})+Bk, & \text{if}\ \chi(\orb)+r+\frac{s}{2}< 0,
\end{cases}
\end{equation}
where $A$ and $B$ are positive universal constants. 
\end{thm}

There are several settings in which the upper bound in Theorem~\ref{h-confinv0} can be simplified.   {An important example is when  
$(\orb, g)$ conformally embeds as a domain in}  the quotient of Euclidean space or the round half-sphere by a finite group $\Gamma$ of isometries with $|\Gamma|=q$. Then $\iso_g(\orb)\geq C{q}^{-\frac{1}{n}}$ where $C$ is a {positive} constant depending only on the dimension.  Thus the bound in Theorem~\ref{h-confinv0} reduces to
  \begin{equation}\label{intro-growth}\sigma_k(\orb,g)\vol(\partial\orb)^{\frac{1}{n-1}}\le c(n) {{q}^{\frac{1}{n}\left(1-\frac{1}{n-1}\right)}} {k^{\frac{2}{n}}}, \end{equation}
  where $c(n)$ is a constant depending only on the dimension.   This result was known previously \cite[Thm. 1.2]{CEG} in the special case that $\Gamma$ is trivial, i.e., $q=1$.  For further results of this flavor, see \S \ref{sec:bounds}.

Finally, we note that one can use similar methods to obtain upper bounds for the Laplace eigenvalues of a compact {orbifold} analogous to Laplace {eigenvalue} bounds already known in the  manifold setting \cite{CEG,H}.  \\

The paper is organized as follows. In \S \ref{sec:ofld}, we recall the necessary background information on Riemannian orbifolds. Then in \S \ref{sec:psido}, we review the notion of pseudifferential operators on orbifolds and confirm that the Steklov problem extends to orbifolds.  In \S \ref{sec:invtomog}, we prove Theorem~\ref{cone-intro}.  The main results concerning the Steklov spectrum, alluded to in \S\S \ref{asymp}, \ref{isosp}, and \ref{bound} above, are addressed in \S \ref{sec:asympt}, \ref{sec:examples} and \ref{sec:bounds}, respectively.

\subsection*{Acknowledgements}  {We are delighted to thank David Sher and Emilio Lauret for invaluable conversations.  Conversations with David Sher were influential in our study of the Steklov spectrum of orbisurfaces.   Emilio Lauret helped us to discover Example~\ref{h-ex}, which established the sharpness of the upper bounds as described in Theorem~\ref{h-confinv0}.

In our original version of this article, Theorem~\ref{cone-intro} only stated that the cone and disk have the same Steklov spectrum.  We are very grateful to the referee for pointing out that our computations actually give the stronger result that the Dirichlet-to-Neumann operators coincide.

We thank the Banff International Research Station for hosting the 2015 Women in Geometry Workshop (15w5135) where this work began.  In addition we gratefully acknowledge the Max Planck Institute for Mathematics in Bonn for its support and hospitality, enabling the six of us to reunite and advance our collaboration.    

The first  author is supported by D.G.I.~(Spain) and FEDER Project MTM2016-77093-P, by Junta de Extremadura and FEDER funds.  The second author is partially supported by a grant from the Simons Foundation (210445 to Emily B. Dryden). The third and fifth authors thank the National Science Foundation (DMS-1632786) for providing travel support to the 6th Workshop on Differential Geometry in Argentina, and the third author also thanks the University of C\'ordoba for hosting an extended stay; there she conferred with Emilio Lauret as described above. The fourth author was supported by the Max Planck Institute for Mathematics in Bonn and is grateful for its support and excellent working conditions.

%%%%%%%%%%%%%%%%%%%%%%%%%%%%%%%%%%%%%%%%%%%%%%%%%%%%%%%%%%%%%%%%%%%%%%%%%%%%%%%%%%%%%%%%%%%%%%%%%%%%%%%%%%%%%%%%%%%%%%%%%%%%%%%%%%%%%%%%%%%%%%%%%%%%%%

\section{Riemannian orbifold background}\label{sec:ofld}

\subsection{Definitions and basic properties}

We follow the presentation in \cite{gordon12}, which the reader may consult for further details. We begin by defining a Riemannian orbifold and highlighting several important features of orbifolds.  We will simultaneously define the notions of orbifold and orbifold with boundary.

\begin{definition} \label{defn:ofld}\normalfont
Let $\orb$ be a second countable Hausdorff space.  
 \begin{enumerate}
\item An $n$-dimensional \emph{orbifold coordinate chart}\footnote{An alternate term often used for orbifold coordinate chart is \emph{uniformizing system}.} (or just \emph{orbifold chart}) over an open subset $U\subset\orb$ is a triple $\cc$ for which: $\wtu$ is a connected open subset of $\R^n$, $\Gamma_U$ is a finite group acting effectively on $\wtu$ by diffeomorphisms, and the mapping $\varphi_U$ from $\wtu$ onto $U$ induces a homeomorphism from the orbit space $\Gamma_U\backslash \wtu$ onto $U$.  When defining orbifolds with boundary, one replaces $\R^n$ by $\R^n_+=\{(x_1, x_2, \dots, x_n): x_n\ge 0\}$.
\item An \emph{orbifold atlas} is a collection of compatible orbifold charts $\cc$ such that the images $\varphi_U({\wtu})$ cover $\orb$.  (See \cite[Defn. 1.4(1)]{gordon12} for the definition of compatibility.)  An \emph{orbifold} is a second countable Hausdorff space together with an orbifold atlas.  In the case of orbifolds with boundary, the \emph{boundary} of $\orb$ consists of all points $p\in\orb$ such that relative to some (and hence every) orbifold chart  $\cc$ with $p \in U$, the inverse image $\varphi_U^{-1}(p)$ lies in $\{x: x_n=0\}$.  Two-dimensional orbifolds will be called \emph{orbisurfaces}. 
\item An orbifold is said to be \emph{good} if it is globally the quotient $\Gamma\bs M$ of a manifold $M$ by a discrete group $\Gamma$ acting properly discontinuously.  Otherwise it is said to be \emph{bad}.
\item For $p\in \orb$, let $\cc$ be an orbifold chart on a neighborhood $U$ of $p$.  The \emph{isotropy type} of $p$, denoted $\operatorname{Iso}(p),$ is the isomorphism class of the isotropy group of a lift $\tilde p$ of $p$ in $\wtu$ under the action of $\Gamma_U$.  The isotropy type of $p$ is independent of the choice of lift $\tilde p$ as well as the choice of orbifold chart.  The isotropy type is canonically identified with a conjugacy class of subgroups of $O(n)$.    (For details see \cite[\S 1.2]{gordon12}.)
\item Points in $\orb$ with nontrivial isotropy are called \emph{singular points}.  Points that are not singular are called \emph{regular points}. 
\item An orbifold chart $\cc$ is said to be \emph{orientable} if the group $\Gamma_U$ consists of orientation-preserving transformations of $\wtu$.  In that case, an orientation of $\cc$ is given by a choice of orientation on $\wtu$. An \emph{orientable} orbifold is one which admits an atlas of compatibly oriented charts. 
\end{enumerate}
\end{definition}

\begin{remark} \label{covering}\normalfont
\begin{enumerate}
\item An orbifold is a stratified space with strata consisting of connected sets of points of like isotropy type.  The set of singular points in $\orb$, or \emph{singular set}, is a set of measure zero.  Correspondingly, the stratum consisting of all regular points has full measure.  Singular strata of codimension one are called \emph{reflectors} or \emph{mirrors}, since locally they are quotients of open sets in $\R^n$  (or $\R^n_+$) by the group generated by a reflection.  In particular, a reflector always has isotropy type $\Z_2$. 
\item The boundary of an orbifold is itself an orbifold without boundary.
It is important to distinguish between the boundary of the underlying second countable Hausdorff space {and the orbifold boundary as defined in Definition \ref{defn:ofld}(2).}   The latter is contained in the former, but the containment may be proper.  More precisely, the boundary of the underlying topological space consists of the orbifold boundary together with all reflectors; the latter lie in the interior of the orbifold.  For example, the quotient $\orb=\Z_2\bs\R^2$, where $\Z_2$ acts by reflection across the $x$-axis, is a good orbifold without boundary although the underlying topological space is homeomorphic to $\R^2_+$.
\item Orbifolds that contain reflectors are never orientable.   By doubling $\orb$ across all reflectors, one obtains a two-fold covering orbifold of $\orb$ all of whose singular strata have codimension at least two.   We refer to \cite[Defn. 1.28]{gordon12} for the definition of orbifold covering map. 

\end{enumerate}
\end{remark}

\begin{defn}\label{def.riem}\normalfont A Riemannian structure $g$ on an orbifold is defined by giving the local cover $\wtu$ of each orbifold chart $\cc$ a $\Gamma_U$-invariant Riemannian metric in such a way that the maps involved in the compatibility condition are isometries.  An orbifold with a Riemannian structure will be called a \emph{Riemannian orbifold}. 

Given a Riemannian metric on $\orb$, the (sectional, Ricci, or scalar, respectively) \emph{curvature} at a point $p\in\orb$ is defined to be the (sectional, Ricci, or scalar, respectively) curvature at any lift $\tilde{p}$ of $p$ in any orbifold chart $\cc$ about $p$.  The curvature is independent of the choice of chart and of lift.
\end{defn}

We end this subsection by recalling the definition of a smooth map on an orbifold and the notion of a suborbifold, again following \cite{gordon12}.

\begin{definition}\label{def.smoothmap} \normalfont Let $\orb$ and $\OP$ be orbifolds.  Suppose a function $f:\orb \rightarrow \OP$ is continuous with respect to the underlying space topologies of $\orb$ and $\OP$.  We say $f$ is a \emph{smooth orbifold map} if for every $p \in \orb$, there exist neighborhoods $U$ about $p$ and $V$ about $f(p)$ with $f(U)\subset V$ and charts $\cc$ over $U$ and $(\widetilde V, \Gamma_V, \vp_V)$ over $V$ for which the following two conditions hold: 
\begin{enumerate}
\item[i.] $f|_U$ lifts to a smooth map $\widetilde f: \wtu \rightarrow \widetilde V$ satisfying $\vp_V \circ \widetilde f = f \circ \vp_U$, and
\item[ii.]  there exists a homomorphism $\psi : \Gamma_U \rightarrow \Gamma_V$ such that for all $\gamma \in \Gamma_U$, we have $\widetilde f \circ \gamma = \psi(\gamma) \circ \widetilde f$.
\end{enumerate}
\end{definition}

\begin{definition}\normalfont  Let $\orb$ and $\OP$ be orbifolds, possibly with boundary, and let $i: \orb \rightarrow \OP$ be a smooth orbifold map such that $i:\orb \rightarrow i(\orb)$ is a homeomorphism with respect to the subspace topology on the image.  We will usually identify $\orb$ with $i(\orb)$.  We will say that $\orb$ is a \emph{suborbifold} of $\OP$ if the local lifts $\tilde{i}$ as in Definition~\ref{def.smoothmap} are embeddings.
\end{definition}

\begin{remark}\normalfont  A class of suborbifolds that will be important in what follows are those obtained as subdomains of larger ambient orbifolds.
\end{remark}

\subsection{Examples of orbifolds}\label{orbifold examples}

\begin{ex}\label{exa:dim1ofds} (The one-dimensional compact orbifolds) In a one-dimensional orbifold, all singular strata must have codimension one.  Thus by Remark~\ref{covering}, all singularities are reflector points with $\Z_2$ isotropy. Hence the one-dimensional compact orbifolds consist of the circle (all isotropy is trivial), a segment with both endpoints reflector points, a segment with one endpoint a reflector point, and a segment with no reflector points.  The first two of these are closed orbifolds and the latter two are orbifolds with boundary.  All these orbifolds are good: a segment with two reflector points is the quotient of a circle by the group generated by a reflection symmetry, and a segment with one reflector point is the quotient of a line segment of twice the length by a reflection.
\end{ex}

\begin{remark}\label{typeI_II}\normalfont  The boundary of a compact two-dimensional orbifold consists of finitely many closed one-dimensional orbifolds.   As just observed, a closed one-dimensional orbifold is of one of two types:  a circle or the quotient of a circle by a reflection.  We will refer to these two types of boundary components as \emph{type} I and \emph{type} II, respectively.  
\end{remark}

\begin{ex}\label{exa:flat cone}
A cone orbifold with cone angle $\frac{2\pi}{k}$ is obtained by taking the quotient of a disk $D\subset \R^2$ by a cyclic group of symmetries generated by a rotation through angle $\frac{2\pi}{k}$.  The point fixed by the rotations is an interior singular point called a \emph{cone point} of order $k$. The circular boundary of this orbifold, which consists entirely of regular points, is the image under the quotient of the boundary of $D$.  
\end{ex}

\begin{ex} Suppose $\orb$ is an orbisurface whose only singular points are cone points and whose underlying topological space is a sphere. If $\orb$ has a single cone point of order $k$ then it is called a $k$-\emph{teardrop}.  If $\orb$ has two cone points of orders $p$ and $q$ then it is called a $(p,q)$-\emph{football}.  If $\orb$ has three cone points of orders $p$, $q$, and $r$ then it is called a $(p,q,r)$-\emph{pillow}.  All teardrops are bad orbifolds, as are footballs for which $p \ne q$.  When $p=q$, a football is the quotient of a sphere under a cyclic group of rotations.   All pillows are good orbifolds.
\end{ex}

\subsection{Orbifold bundles}  An orbibundle consists of: 
\begin{itemize}
\item an orbifold $\mathcal E$ (the total space) and an orbifold $\orb$ (the base space),
\item a surjective map $\pi_{\mathcal E}:{\mathcal E}\to \orb$ (the bundle projection),
\item a manifold $F$ (the model fiber),
\item a subgroup $G$ of Diff$(F)$ (the structure group), and
\item  a collection of mutually compatible $F$-bundle orbifold charts for ${\mathcal E}$ over $\orb$ whose images cover $\orb$.   (The $F$-bundle orbifold charts play the role of the local trivializations in the familiar definition of bundles over manifolds.  We omit the definition here; see \cite[\S 2.3]{gordon12} for details.)
\end{itemize}

The only orbibundle that will arise explicitly in this article is the orthonormal frame bundle of a Riemannian orbifold.   We gather here all the properties that we will need.

\begin{defn}[Orthonormal frame bundle]\label{def.onfbundle}\normalfont First consider a good Riemannian orbifold $\orb=\Gamma\bs M$, where $M$ is an $n$-dimensional Riemannian manifold and $\Gamma$ is a discrete group of isometries of $M$ acting with finite isotropy groups.  Let $\F M$ be the orthonormal frame bundle of $M$.  The group $\Gamma$ acts on $\F M$ as the differentials of the isometries in $\Gamma$.   Since an isometry is uniquely determined by its value and differential at any point, the action of $\Gamma$ on $\F M$ is free.  Thus $\F\orb:=\Gamma\bs \F M$ is a manifold.  The bundle projection map $\pi_{\F M}: \F M\to M$ induces a surjective map $\pi_{\F\orb}: \F \orb\to \orb$. This is the orthonormal frame bundle of $\orb$.   (Note that this definition still makes sense if $\orb$, and hence $M$, have boundary.)

In the case of a bad orbifold (possibly with boundary), observe that the image $U$ of any orbifold chart $\cc$ is a good orbifold $\Gamma_U\bs {\widetilde U}$.   The orthonormal frame bundle $\F\orb\to \orb$ is defined in such a way that the restriction to the image $U$ of any orbifold chart is the orthonormal frame bundle $\F U=\Gamma_U\bs \F{\widetilde U}\to U$ as defined above.   We omit the details here.

\end{defn}

\begin{remark}\label{rem.onf}\normalfont (i) The orthonormal frame bundle $\F\orb$ is a principal orbibundle with model fiber $O(n)$ and structure group $O(n)$.  The total space $\F\orb$ is actually a smooth manifold.  The orbifold $\orb$ is the quotient of $\F\orb$ by the action of $O(n)$. 

(ii) One difference between an orbibundle and a bundle over a smooth manifold is that, although the generic fiber $\pi_{\mathcal E}^{-1}(x)$ is diffeomorphic to $F$, there may be singular fibers that are diffeomorphic to the quotient of $F$ by a finite group action.  In the case of the orthonormal frame bundle of a Riemannian orbifold $\orb$, the fiber over each regular point of $\orb$ is diffeomorphic to $O(n)$, but the fiber over a point $p\in \orb$ with isotropy group $\operatorname{Iso}(p)\subset O(n)$ is diffeomorphic to $\operatorname{Iso}(p)\bs O(n)$.

\end{remark}

\begin{prop}\label{m/g} Every orbifold $\orb$, good or bad, can be realized as the orbit space of an effective action of a compact Lie group $G$ on a manifold $M$ with all isotropy groups finite.  If $\orb$ has boundary, then so does $M$ and $\partial\orb = G\bs\partial M$.

\end{prop}
For instance, we may take $M=\F\orb$ with respect to some choice of Riemannian metric on $\orb$, and take $G=O(n)$.

\begin{remark}\label{metric_onf}\normalfont Let $G\bs M$ be a realization of $\orb$ as in Proposition~\ref{m/g}.  Given a Riemannian metric $g$ on $\orb$, we may give $M$ a Riemannian metric for which the projection $\pi_M: M\to \orb=G\bs M$ is a Riemannian submersion.  We can further require that the Riemannian metric on $M$ restrict to the bi-invariant metric of volume one on each fiber.   More precisely, the metric on the regular fibers is the bi-invariant metric of volume one, and the metric on the singular fibers (see Remark~\ref{rem.onf}) is the metric induced by the bi-invariant metric of volume one on $G$.   \end{remark}

\subsection{Volume comparison on complete orbifolds}\label{geodesics}

A Riemannian orbifold will be called \emph{complete} if it is complete as a {length} space with respect to the metric induced by the orbifold's Riemannian structure.  The Hopf-Rinow Theorem for length spaces (see \cite[p. 9]{Gro99}) implies that if an orbifold is complete in this sense, then
any two points in the orbifold can be joined by a length-minimizing curve.  Note that this implies that the set of regular points of an orbifold forms a convex set.

\begin{remark}\normalfont There are some subtleties in the definition of geodesics on orbifolds. J. Borzellino \cite{Bz} showed that if a length-minimizing curve in an orbifold is not entirely contained
within the singular set, it can only intersect the singular set at its end points.  Although from the perspective of the geodesic flow of an orbifold it makes sense to continue a geodesic across the singular set, we then lose the property that the geodesic is locally length minimizing.  This can be seen in Example \ref{exa:flat cone} if you consider two points on opposite sides of the cone point.  The two points are more efficiently connected by a geodesic that goes around the cone point rather than the path between them that goes directly through it.  In what follows our interest is in length-minimizing curves and we will avoid use of the word `geodesic.'
\end{remark}

Borzellino \cite[Prop. 20]{Bz} showed that the Bishop-Gromov Comparison Theorem is valid for orbifolds.   As in Definition~\ref{def.riem}, 
we say that an $n$-dimensional Riemannian orbifold $(\orb,g)$ has $\ric_{(\orb,g)}\ge \kappa(n-1)$, $\kappa\in\R$, if for each orbifold chart $\cc$, $\wtu$  is a Riemannian manifold of Ricci curvature greater than or equal to $ \kappa(n-1)$. 

\begin{prop}[Relative Volume Comparison Theorem for Orbifolds \cite{Bz}]\label{b-g} Let $(\orb,g)$ be a complete {n-dimensional} Riemannian orbifold without boundary and  $Ric_{(\orb,g)} \geq \kappa(n-1)$. Then for every $p\in\orb$, the function $f_p(r)=\frac{\vol(B(p,r))}{v(n,{\kappa},r)}$ is nonincreasing on $(0,\infty)$, where $v(n,{\kappa},r)$ is the volume of a ball of radius $r$ in the $n$-dimensional simply connected space form of constant curvature $\kappa$. Moreover, $\lim_{r\to0} f(r)=\frac{1}{|\operatorname{Iso}(p)|}$, where $|\operatorname{Iso}(p)|$ is the order of the isotropy type at $p$. 
\end{prop}

Since the regular set of $\orb$ is convex, for a regular point $p$ one can follow the argument in the  manifold setting \cite[p. 279]{pet16}.  For a singular point $p$, one can consider a sequence of regular points $\{p_i\}$ with $p_i\to p$. The proposition follows from the fact that $\vol(B(p_i,r))\to{\vol(B(p,r))}$ (see \cite{Bz} for more details).

\subsection{Some tools and results for orbisurfaces}

Much work has been done to extend results from the manifold setting to the orbifold setting.  Thurston \cite{Th} defined the notion of Euler characteristic for orbifolds.

\begin{defn}[Euler characteristic]\label{def:Euler_characteristic}\normalfont Let $\{c_i\}$ be a cell division of an orbifold $\orb$ for which the isotropy group associated to the interior points of each cell is constant. The Euler characteristic of $\orb$ is defined by $$\chi(\orb):=\sum_i (-1)^{dim \, (c_i)} \frac{1}{|\operatorname{Iso}(c_i)|}$$ where $|\operatorname{Iso}(c_i)|$ is the order of the isotropy type associated to the cell $c_i$.
\end{defn}

\begin{remark}\label{rem.euler}\normalfont (i) If $\widetilde \orb$ is a $k$-sheeted orbifold cover of $\orb$, then $\chi(\widetilde \orb)=k\chi(\orb)$.

(ii) Observe that the Euler characteristic need not be an integer.   For example, the cone in Example~\ref{exa:flat cone} has Euler characteristic $\frac{1}{k}$.

\end{remark}

The following extension of the Gauss-Bonnet Theorem is treated in detail by I. Satake in \cite{Satake57}.  Although Satake proves the theorem for general dimension, we will only need the two-dimensional case.

\begin{thm}[Gauss-Bonnet Theorem for Orbisurfaces]\label{thm:gauss_bonnet} Let $\orb$ be a compact Riemannian orbisurface and let $K$ denote its Gaussian curvature. Then 
 $$\int_{\orb} K d\vol_{\orb} = 2\pi \chi(\orb).$$ 
\end{thm}

As shown in \cite[Thm. 13.3.6]{Th} the orbifold Euler characteristic gives us a convenient way to classify compact closed orbisurfaces.  The different classes correspond to different metric behavior; to formalize this idea, we need a definition.

\begin{definition}\label{def:conf_class}\normalfont  The conformal class $[g]$ of a metric $g$ on an orbifold $\orb$ is given by
\[[g]:=\{e^fg: f \in C^0(\orb) \ \text{and} \ f|_{\operatorname{Int}({\orb})} \in C^\infty (\operatorname{Int}(\orb))\}\]
\end{definition}

\begin{prop}\label{uniform}  \ 
\begin{enumerate}
\item The only bad closed orbisurfaces are the teardrop, $(p,q)$-footballs with $p \ne q$, and quotients of these by reflections.  All bad closed orbifolds have positive Euler characteristic.\item Let $\orb$ be a closed orbisurface and let $g$ be any Riemannian metric on $\orb$.  
\begin{enumerate}
\item If $\chi(\orb)<0$, then $g$ is conformally equivalent to a Riemannian metric of constant negative curvature.
\item If $\chi(\orb)=0$, then $g$ is conformally flat.
 \item If $\chi(\orb)>0$, then $g$ is conformally equivalent to a Riemannian metric $h$ of positive curvature.  If, moreover, $\orb$ is a good orbisurface, then $h$ can be taken to have constant positive curvature.
\end{enumerate}
In all three cases, normalized Ricci flow carries $g$ to a conformally equivalent metric with the indicated property.
\end{enumerate}
\end{prop}

{
We summarize the proof of this well known result: 
\begin{proof}  Statement (1) follows from \cite[Thm.~13.3.6]{Th}.  
For statements (2a) and (2b), $\orb$ is necessarily a good orbisurface (see \cite{PeterScott}).  As shown in \cite[Thm 2.5]{PeterScott}, every good closed orbisurface is finitely covered by a surface.   Let $M$ be a finite surface cover of $\orb$, say $\orb=\Gamma\bs M$, and pull $g$ back to a $\Gamma$-invariant Riemannian metric  $\tilde g$ on $M$.   The normalized Ricci flow on $M$ carries $\tilde g$ to a constant curvature metric $\tilde h$.   Since the Ricci flow preserves conformal classes in dimension 2 and also preserves isometries, $\tilde h$ is  $\Gamma$-invariant and descends to a constant curvature metric on $\orb$ conformally equivalent to $g$.  By the Gauss--Bonnet theorem the curvature has the same sign as $\chi(\orb)$ and statements (2a) and (2b) follow.    Statement (2c) also follows in case $\orb$ is good.

If $\orb$ is a bad closed orbisurface, then B. Chow and L.-F. Wu showed in \cite[Thm. 1.2]{ChowWu} that the normalized Ricci flow on $\orb$ carries every metric to a soliton metric; Wu \cite{Wu} showed that this soliton metric is unique and has positive curvature.  This yields (2c) and also proves the final statement of the proposition in the case of bad orbisurfaces.

\end{proof}
}

%%%%%%%%%%%%%%%%%%%%%%%%%%%%%%%%%%%%%%%%%%%%%%%%%%%%%%%%%%%%%%%%%%%%%%%%%%%%%%%%%%%%%%%%%%%%%%%%%%%%%%%%%%%%%%%%%%%%%%%%%%%%%%%%%%%%%%%%%%%%%%%%%%%%%%

\section{Establishing the Steklov problem on orbifolds}\label{sec:psido}

\subsection{Pseudodifferential operators on orbifolds}\label{pseudodiff}

We review the definition and properties of pseudodifferential operators on Riemannian orbifolds.  A detailed treatment can be found in \cite{Bu08}, \cite{GN1} and \cite{GN2}; see also \cite{SU}.  The {book by M. A. Shubin} \cite{Sh} provides a comprehensive treatment of pseudodifferential operators in general. Here we will primarily follow B.~Bucicovschi \cite{Bu08} and A.~Uribe and the sixth author \cite{SU}.  
 
A pseudodifferential operator on a Riemannian orbifold may be defined either through the use of orbifold charts, or globally by realizing the orbifold as a quotient of a Riemannian manifold by the action of a compact Lie group.   The following definition gives the first approach.

\begin{defn}\label{def_pseudodiff}\normalfont  Let $(\orb,g)$ be a Riemannian orbifold.   A linear map $A:C^\infty(\orb)\to C^\infty(\orb)$ is said to be a \emph{pseudodifferential operator of order} $m$ if for each orbifold chart $(\Ut, \Gamma_U,\varphi_U)$, there exists a ${\Gamma_U}$-equivariant (i.e., that commutes with the $\Gamma_U$-action) pseudodifferential operator $\At:C^\infty(\Ut)\to C^\infty(\Ut)$ of order $m$ such that $\At (\varphi_U^* f)=\varphi_U^*(A f)$ for all $f\in C^{\infty}(U)$.   In this case, the operator $\At$ is not uniquely defined but is shown in \cite{Bu08} to be unique up to a smoothing operator.  The \emph{classical} pseudodifferential operators on $\orb$ are those for which the $\At$ can be chosen to be classical (in the sense described in \cite[\S 3.7]{Sh}).
\end{defn}

The following proposition gives the global approach.  Recall from Proposition~\ref{m/g} that every orbifold $\orb$ is a quotient of a manifold $M$ by the action of a compact Lie group $G$ acting smoothly and effectively on $M$ with all isotropy groups finite.  Following Remark~\ref{metric_onf}, given a metric $g$ on $\orb$ we can construct a metric on $M$ so that the bundle projection is a Riemannian submersion and the metric restricted to fibers is the bi-invariant metric of volume one on $G$.

\begin{prop}\label{prop.global} Let a compact Lie group $G$ act smoothly and effectively by isometries on a Riemannian manifold $M$ and assume all isotropy groups are finite.  Let $\orb=G\backslash M$. 
Then a linear operator $A: C^\infty(\orb)\to C^\infty(\orb)$ is a pseudodifferential operator of order $m>0$ if and only if there exists a $G$-equivariant pseudodifferential operator $\At: C^\infty(M)\to C^\infty(M)$ of order $m$ such that $A$ is the restriction of $\At$ to the $G$-invariant functions.  If $A$ is classical, then $\At$ can be chosen to be classical.  Moreover:

\begin{enumerate}
\item[(i)] If $A$ is elliptic, $\widetilde A$ can be chosen to be elliptic, and if $\widetilde A$ is elliptic then $A$ is automatically elliptic.
 \item[(ii)] Give $M$ a Riemannian metric that restricts to a bi-invariant metric on the $G$-orbits and such that the projection $M\to G\backslash M=\orb$ is a Riemannian submersion.  If $A$ is positive, respectively bounded below, and symmetric, then $\At$ can be chosen to be positive, respectively bounded below, and symmetric as well.
 \end{enumerate}
\end{prop}
The only if statement is proven in \cite{Bu08} and the if statement in \cite{SU}.  For statements $(i)$ and $(ii)$, see the proof of  \cite[Thm. 3.5]{Bu08}.

\begin{prop}[General Spectral Theorem]\label{spec_thm}  Let $\orb$ be a compact Riemannian orbifold and $A$ an elliptic, symmetric, positive pseudodifferential operator on $\orb$ of order {$m > 0$}.  Then $A$ acting on $L^2(\orb)$ is essentially self-adjoint.  There exists an orthonormal basis of $A$-eigenfunctions in $L^2(\orb)$ whose eigenvalues form a discrete subset of $\R$ that is bounded below and diverges to $+\infty$, and each eigenvalue has finite multiplicity.  
\end{prop}

The Spectral Theorem is proven for orbifolds in \cite[Thm. 3.5]{Bu08} by realizing $\orb$ as $G\backslash M$ and applying the results of Proposition~\ref{prop.global}; see also \cite[Prop. 2.4]{SU}.

We conclude our discussion of general pseudodifferential operators on orbifolds by stating a lemma that will be used in \S\ref{sec:asympt}; see also Proposition \ref{spectrumislocal} in this section.

\begin{lemma}\label{lem 2.1} Let $M$ be a compact Riemannian manifold on which a compact Lie group $G$ acts smoothly and effectively by isometries and assume all isotropy groups are finite.   Let $\At_1$ and $\At_2$ be $G$-equivariant elliptic, self-adjoint pseudodifferential operators that are bounded below and have order $m>0$, and suppose that $\At_1-\At_2$ is a smoothing operator.
Let $A_1$ and $A_2$ be the restrictions of $\At_1$ and $\At_2$ to the space $L^2(M)^G$ of $G$-invariant functions.   Then the eigenvalues of $A_1$ and $A_2$ satisfy
\begin{equation}\label{h-smooth}\lambda_j(A_1)-\lambda_j(A_2)=O(j^{-\infty}).\end{equation}
Moreover, the relation \eqref{h-smooth} holds when $A_1$ and $A_2$ are considered as elliptic, self-adjoint pseudodifferential operators on the  compact Riemannian orbifold $\orb=G\backslash M$.\end{lemma}

When $G$ is trivial,  this result is well-known and a careful proof is given in \cite[Lem. 2.1]{GPPS}. In the general case, the proof is   essentially  verbatim. The last part  of the lemma is now an immediate consequence of Proposition~\ref{prop.global}.

\subsection{The Dirichlet-to-Neumann operator in the orbifold setting}
We now define the Dirichlet-to-Neumann operator on a compact {Riemannian} orbifold {$\orb$} with boundary {$\partial\orb$}.  Recall from Remark~\ref{rem.onf}(i) that $\orb$ can be expressed as a quotient of its orthonormal frame bundle $\F\orb$ under the action of the orthogonal group $O(n)$.  We give $\F\orb$ a Riemannian metric as in Remark~\ref{metric_onf} for which the projection $\pi:\F\orb\to\orb$ is a Riemannian submersion and for which the induced metric on each fiber is the metric arising from the bi-invariant metric of volume one on $O(n)$.  

\begin{lemma}\label{harm_extn}  Let $(\orb,g)$ be a compact Riemannian orbifold with boundary and take $u \in C^\infty(\partial \orb)$.  Then there exists a unique harmonic function $Hu \in C^\infty(\orb)$ for which $Hu |_{\partial \orb} = u$.
\end{lemma}

\begin{proof}  With respect to the Riemannian metric defined above on $\mathcal{F}\orb$, the fibers are totally geodesic.   Consequently, we have $$\pi^*\circ \Delta_\orb  = \Delta_{\mathcal{F}\orb}\circ \pi^*.$$  Thus the harmonic functions on $\orb$ are precisely the functions that pullback to harmonic functions on $\mathcal{F}\orb$.  

 Let $v = \pi^*(u) \in C^\infty(\partial(\F\orb))$, and let $Hv$ be the unique harmonic extension of $v$ to $\mathcal{F}\orb$.  The uniqueness of this extension, the $O(n)$-invariance of $v$, and the fact that $O(n)$ acts on $\mathcal{F}\orb$ by isometries together imply that $Hv$ is $O(n)$-invariant.  Hence there exists a unique harmonic function $Hu \in C^\infty(\orb)$ with $\pi^*(Hu) = Hv$.  By our construction of $Hu$, we have that $Hu |_{\partial \orb} = u$.
\end{proof}

Lemma~\ref{harm_extn} allows us to define the Dirichlet-to-Neumann operator on an orbifold $\orb$ with boundary. We first specify what we mean by the normal derivative of a function across the boundary of $\orb$.

\begin{defn}\label{normalderiv}\normalfont
Let $f\in C^\infty(\orb)$.  We  write $\partial_\nu f$ for the normal derivative of $f$ across the boundary of $\orb$, defined at each $p\in \partial\orb$ as follows:   Let $(\Ut, \Gamma_U, \vp_U)$ be an orbifold chart on $\orb$ with $p\in U$.   By Definition~\ref{defn:ofld}(2), $\vp_U^{-1}(\partial \orb \cap U)=B$ where $B:=\Ut\cap \{x:x_n=0\}$ and $n$ is the dimension of $\orb$.  Let ${\tilde f}= \vp_U^*(f)$, and let ${\tilde p} \in  \vp_U^{-1}(p)$.    The Riemannian metric on $U\subset\orb$ corresponds to a $\Gamma_U$-invariant Riemannian metric on $\Ut$.   Set $\partial_\nu f(p)=\partial_\nu{\tilde f}({\tilde p})$, where $\partial_\nu{\tilde f}$ is the normal derivative of $\tilde f$ across $B$.   The fact that $\Gamma_U$ acts isometrically on $\Ut$ guarantees that this definition is independent of the choice of lift $\tilde p$ of $p$.  
\end{defn}

\begin{defn}\label{ofdD2N}\normalfont  The Dirichlet-to-Neumann operator $\D_{(\orb,g)}$ on $C^\infty(\partial \orb)$ is defined by 
\[\D_{(\orb,g)}(u) = \partial_\nu (Hu)\]
where $Hu$ is the unique harmonic extension of $u \in C^\infty(\partial \orb)$ to $\orb$.  We will sometimes omit the subscript $\orb$ and/or $g$ when this will not cause confusion. 
\end{defn}

With this definition in place, the next proposition examines the relationship between the Dirichlet-to-Neumann operator on $\orb$ and that on $\F\orb$.

\begin{prop}\label{D2Ns_commute}  Let $C^\infty(\partial(\F\orb))^{O(n)}$ denote the space of $O(n)$-invariant smooth functions on $\partial(\F\orb)$.   With respect to the Riemannian metric on $\F\orb$ defined in Remark~\ref{metric_onf}, the restriction of the Dirichlet-to-Neumann operator $\D_{\F\orb}$ to $C^\infty(\partial(\F\orb))^{O(n)}$ corresponds to $\D_{(\orb,g)}$ on $C^\infty(\partial\orb)$.  That is, the following diagram commutes:
\[
\begin{array}{ccc}
C^\infty(\partial(\F\orb))^{O(n)} & \xrightarrow{\D_{\F\orb}} & C^\infty(\partial(\F\orb))^{O(n)}\\
\pi^*\uparrow & & \uparrow\pi^*\\
C^\infty(\partial\orb) & \xrightarrow{\D_{(\orb,g)}} & C^\infty(\partial\orb).
\end{array}
\] 
\end{prop}

\begin{proof}   For $u\in C^\infty(\partial \orb)$, define $Hu$, $v$ and $Hv$ as in the statement and proof of Lemma~\ref{harm_extn}.  We need to show that $(\pi_{|\partial(\F\orb)})^*(\partial_\nu Hu)$ coincides with the normal derivative of $Hv$ across the boundary of $\F\orb$.   In case $\orb$ is a smooth manifold, this assertion is trivial since $\pi$ is a Riemannian submersion.  For the general case, let $(\Ut, \Gamma_U, \vp_U)$ be an orbifold chart.   Since $U=\Gamma_U\bs \Ut$ and $\F U  =\Gamma_U\bs \F\Ut$, the assertion is a straightforward consequence of the definition of the normal derivative $\partial_\nu Hu$ as given in Definition~\ref{normalderiv}.
\end{proof}

Applying Proposition~\ref{prop.global}, we immediately obtain the following corollary to Proposition~\ref{D2Ns_commute}.

\begin{cor}  The operator $\D_{(\orb,g)}$ is a first-order, positive, elliptic, self-adjoint pseudodifferential operator. 
\end{cor}

\begin{defn}\normalfont The \emph{Steklov spectrum} of $(\orb, g)$ is the eigenvalue spectrum of $\D_{(\orb,g)}$.   We will denote this spectrum by $\Stek(\orb, g)=\{ 0=\sigma_1 \le \sigma_2 \le \sigma_3 \le {\cdots} \}$, sometimes omitting the name of the orbifold or the metric if no confusion will result.
\end{defn}

Next we present two foundational propositions that will be used in later computations.

\begin{prop}\label{spectrumislocal}    Suppose that $(\orb,g)$ and $(\orb',g')$ are compact Riemannian orbifolds with boundary and that there exists an isometry from $\partial \orb$ to $\partial \orb'$ that extends to an isometry $F$ from a neighborhood of $\partial\orb$ in $\orb$ to a neighborhood of $\partial\orb'$ in $\orb'$.  Then the Steklov eigenvalues satisfy
$$\sigma_j(\orb,g)-\sigma_j(\orb',g')=O(j^{-\infty}).$$
\end{prop}

\begin{proof} The isometry $F$ pulls back to an isometry from a neighborhood of the boundary in $\F\orb$ to a neighborhood of the boundary in $\F\orb'$.   By \cite{LU}, the symbol of the Dirichlet-to-Neumann operator of a Riemannian manifold depends only on the metric in a neighborhood of the boundary.   Thus the Dirichlet-to-Neumann operators of $\F\orb$ and $\F\orb'$ (viewed as operators on $\partial F\orb$, identified via the isometry with $\partial F\orb'$) differ only by a smoothing operator.     Thus the proposition follows from Lemma~\ref{lem 2.1} and Proposition~\ref{D2Ns_commute}.
\end{proof}

\begin{prop}[Variational characterization of eigenvalues]\label{prop.var} Let $(\orb, g)$ be a compact Riemannian orbifold with boundary. For $k=1,2,\dots$, let ${\mathcal E}(k)$ denote the collection of all $k$-dimensional subspaces of the Sobolev space $H^1(\orb)$.   Then the Steklov eigenvalues of $(\orb,g)$ are given by
$$\sigma_k=\min_{E\in{\mathcal E}(k)}\,\max_{0\ne f\in E}\,R_\orb(f)$$ where 
$$R_\orb(f)=\frac{\int_\orb\,|\nabla_g f|^2 \,d\vol_{(\orb,g)}}{\int_{\partial \orb}\,f^2\,d\vol_{(\partial\orb,g)}}.$$
\end{prop}

\begin{proof} In the case of Riemannian manifolds, this variational formula is standard.   In the case of orbifolds, consider the orthonormal frame bundle $\F\orb$ with the Riemannian metric for which $\pi_{\F\orb}:\F\orb\to \orb$ is a Riemannian submersion and the fibers have the metric defined by the bi-invariant metric of volume one on $O(n)$.    Then 
$$R_\orb(f)=R_{\F\orb}(\pi_{\F\orb}^*f).$$   Hence Proposition~\ref{prop.var} follows from the variational formula for the Steklov eigenvalues of
$\F\orb$, restricted to the $O(n)$-invariant functions, together with Proposition~\ref{D2Ns_commute}.
\end{proof}
Note that one can use Green's formula on orbifolds and prove Proposition~\ref{prop.var} directly following the proof as in the manifold setting without lifting the test functions to the frame bundle.

In some special cases, the Steklov eigenvalues and eigenfunctions may be explicitly computed.  We end this section with the computations for the Euclidean ball and the flat half-disk orbisurface.

\begin{ex}[Steklov spectrum of a Euclidean ball]\label{spec.ball} Let $B(0,R)$ be the Euclidean ball of radius $R$ in $\R^n$ centered at the origin.  Each homogeneous harmonic polynomial of degree $m$ is a Steklov eigenfunction with eigenvalue $\frac{m}{R}$.   The spherical harmonics, i.e., the restrictions of the homogeneous harmonic polynomials to the sphere, span the space of $L^2$ functions on the sphere.   Thus the Steklov spectrum of $B(0,R)$ consists precisely of the eigenvalues $\frac{m}{R}$, each with multiplicity given by the dimension of the space of homogeneous harmonic polynomials of degree $m$.

{When} $n=2$, the Steklov spectrum is given by
$$0,\frac{1}{R}, \frac{1}{R}, \frac{2}{R}, \frac{2}{R}, \frac{3}{R}, \frac{3}{R}, {\cdots}$$
and the Steklov eigenfunctions corresponding to $\frac{m}{R}$ are $r^m\cos(m\theta)$ and $r^m\sin(m\theta)$ in polar coordinates.  For our purposes, it will often be more convenient to think in terms of the circumference of the disk instead of its radius.  Let $D$ be a topological disk and let $g_{\ell}^{can}$ be the canonical Euclidean metric on $D$ that makes it into a round disk of circumference $\ell$.  We will use $D(\ell)$ as shorthand for $(D, g_{\ell}^{can})$.  Then
\[
\operatorname{Stek}(D(\ell)) = 0, \frac{2\pi}{\ell}, \frac{2\pi}{\ell}, \frac{4\pi}{\ell}, \frac{4\pi}{\ell}, \frac{6\pi}{\ell}, \frac{6\pi}{\ell}, {\cdots}
\]
\end{ex}

\begin{ex}[Steklov spectrum of the flat half-disk orbisurface]\label{half_disk_orb} 
The \emph{half-disk} orbisurface, which we will denote by $\ohd$, is obtained by taking the quotient of a disk $D$ by a reflection.  This orbifold has a mirror edge of reflector points with $\Z_2$ isotropy type and a single type II boundary component.   The metric $g_{2\ell}^{can}$ on $D(2\ell)$ descends to give a metric $\bar{g}_{\ell}^{can}$ on $\ohd$.  We will use $\ohd(\ell)$ to denote $(\ohd, \bar{g}_{\ell}^{can})$, the flat half disk orbisurface with boundary of length $\ell$.

The Steklov eigenfunctions on $\ohd(\lc)$ pull back to the even spherical harmonics on $D(2\ell)$.  Thus 
$$
\operatorname{Stek}(\ohd(\lc)) = 0,\frac{\pi}{\lc}, \frac{2\pi}{\lc}, \frac{3\pi}{\lc}, {\cdots}
$$
Note that each eigenvalue has multiplicity one.   Comparing with \cite[Thm. 1.4]{GPPS}, we see that $\ohd(\lc)$ cannot be Steklov isospectral to a smooth surface.  
\end{ex}

The Steklov isospectrality of the orbifolds in Example~\ref{ell1ell2} in the introduction follows immediately from Examples~\ref{spec.ball} and \ref{half_disk_orb}.

%%%%%%%%%%%%%%%%%%%%%%%%%%%%%%%%%%%%%%%%%%%%%%%%%%%%%%%%%%%%%%%%%%%%%%%%%%%%%%%%%%%%%%%%%%%%%%%%%%%%%%%%%%%%%%%%%%%%%%%%%%%%%%%%%%%%%%%%%%%%%%%%%%%%%%

\section{The inverse tomography problem for orbifolds}\label{sec:invtomog}

Calder\'on's inverse tomography problem asks whether one can determine the conductivity of a medium from voltage and current measurements on the boundary.  See the survey paper \cite{U} and references therein.  As we will review at the end of this section, Calder\'on's problem is closely related (see \cite{LU}) to the following geometric problem:

\begin{quote}
{\emph{Identifiability problem}.}  Given a closed manifold $N$ (not necessarily connected), consider the set $\mathcal{M}$ consisting of all compact connected  Riemannian manifolds $(M,g)$ with boundary $N$.   For each element $(M,g)$ of $\mathcal{M}$, let $\mathcal{D}_{(M,g)}: C^\infty(N)\to C^\infty(N)$ be the associated Dirichlet-to-Neumann map.   To what extent does $\mathcal{D}_{(M,g)}$ determine the topology and geometry of $(M,g)$?
\end{quote}

\noindent \emph{Trivial changes of metric.}  If $(M,g)\in \mathcal{M}$ and $\psi$ is a diffeomorphism of $M$ with $\psi\equiv 1$ on $N=\partial M$, then trivially $\mathcal{D}_{(M,g)}=\mathcal{D}_{(M,\psi^*g)}$.   In dimension two, the same conclusion holds if we further multiply the metric on $M$ by a conformal factor $e^f$, with $f\equiv 0$ on $N$.  We will refer to such changes of metric as trivial changes.

\smallskip

Lassas and Uhlmann \cite{LassasUhlmann} completely solved the identifiability problem in dimension two:  the Dirichlet-to-Neumann operator determines $M$ and determines the metric $g$ on $M$ up to trivial changes. In all higher dimensions, they obtained the same result within the class of all real analytic manifolds $(M,g)$ with real analytic boundary $N$.  Their results are in fact stronger than stated here, as they only required partial knowledge of the Dirichlet-to-Neumann map.  See also \cite{B}.

We now consider the generalization of this problem to orbifolds.  The theorem below contrasts with the results of Lassas-Uhlmann.

\begin{thm}\label{thm.cone}
For $k=1,2,\dots$, let $D(2\pi k)$ denote the flat disk of circumference $2\pi k$ in $\R^2$. Let $\mathcal{C}_k$ be the cone $\mathcal{C}_k=\Z_k\bs D(2 \pi k)$ where the action of $\Z_k$ on $D(2 \pi k)$ is generated by rotation through angle $\frac{2\pi}{k}$ about the origin.  Then the disk $D(2\pi)$ and the cone $\mathcal{C}_k$ give rise to the same Dirichlet-to-Neumann operator on the circle:    
$$\D_{D(2\pi)}=\D_{\mathcal{C}_k} \mbox{\,\,for\,every\,\,}k=1,2,\dots.$$
\end{thm}

 \begin{proof} It suffices to show that the Dirichlet-to-Neumann operators of $D(2\pi)$ and $\mathcal{C}_k$ have exactly the same spectrum and, for each eigenvalue, the same eigenfunctions.  (Here, the eigenfunctions are functions on the \emph{circle}, as opposed to what we have been referring to as Steklov eigenfunctions, which are the harmonic extensions of the Dirichlet-to-Neumann eigenfunctions.)      By Example~\ref{spec.ball}, we have
 $$\operatorname{Stek}(D(2\pi)) = 0, 1,1, 2,2,3,3,\dots,$$
 and the Dirichlet-to-Neumann eigenspace corresponding to the eigenvalue $j$ is spanned by $\cos(js)$ and $\sin(js)$, where $s\,\,(=\theta)$ is the arclength coordinate on the circle.
 
The Steklov eigenfunctions of $\mathcal{C}_k$ pull back to  $\Z_k$-invariant Steklov eigenfunctions on $D(2\pi k)$.  By Example~\ref{spec.ball}, these are precisely the constant function and the eigenfunctions $r^m\cos(m\theta)$ and $r^m\sin(m\theta)$ as $m=jk$ varies over all positive integer multiples of $k$.  Since $R= k$ in the notation of Example~\ref{spec.ball}, the eigenvalue associated with the eigenspace spanned by $r^{jk}\sin(jk\theta)$ and  $r^{jk}\cos(jk\theta)$ is $\frac{jk}{k}=j$.  In particular,
 $$\operatorname{Stek}(\mathcal{C}_k) = 0, 1,1, 2,2,3,3,\dots$$
 Again letting $s$ denote the arclength coordinate (suitably initialized) on the boundary circle of $\mathcal{C}_k$, the covering map from the boundary circle $r=k$ of $D(2\pi k)$ to  $\partial\mathcal{C}_k$ is given by $(k,\theta)\mapsto s=k\theta$.  Thus the Dirichlet-to-Neumann eigenspace of $\D_{\mathcal{C}_k}$ for the eigenvalue $j$ is spanned by $\sin(js)$ and  $\cos(js)$.  The theorem follows.
  \end{proof}

The proof of Theorem~\ref{thm.cone} used the fact that the dimension of the space of homogeneous harmonic polynomials of degree $m$ on $\R^2$ is independent of $m$.  The analogous statement fails in higher dimensions.

By Remark~\ref{rem.euler}, the Euler characteristic of the cone $\mathcal{C}_k$ is $\frac{1}{k}$. Hence, we conclude:

\begin{cor}\label{orb_isosp_man} Within the class of all Riemannian orbisurfaces:
\begin{enumerate}
\item The Dirichlet-to-Neumann  map does \emph{not} always detect the presence or type of interior singularities.  In particular, an orbisurface with interior singularities but smooth boundary $N$ can have the same Dirichlet-to-Neumann map as a smooth surface with boundary $N$.
\item The Dirichlet-to-Neumann map does not determine the Euler characteristic of an orbisurface.

\end{enumerate}
\end{cor}

Greenleaf, Lassas and Uhlmann \cite{GLU} constructed counterexamples to the identifiability problem in the setting of smooth domains with \emph{singular metrics}.    On the other hand, in Theorem~\ref{thm.cone}, the metrics are smooth (in fact, Euclidean) but we have one smooth domain and one orbifold with an \emph{orbifold singularity}.

We now recall the relationship between the identifiability problem above and the inverse tomography problem.   Let $\Omega$ be a bounded domain in $\R^n$.   For $\gamma$ a positive-definite matrix-valued function on $\R^n$ (the anisotropic conductivity), the inverse tomography problem asks whether one can recover $(\Omega, \gamma)$ from the voltage-to-current map $\Lambda_\gamma:C^\infty(\partial \Omega)\to C^\infty(\partial \Omega)$ given by 
$$\Lambda_\gamma(u)=\langle\gamma \nabla \tilde{u}, \nu\rangle|_{{\partial \Omega}}$$
where $\langle\cdot,\cdot\rangle$ is the Euclidean inner product, $\nu$ is the outward pointing unit normal vector field to $\partial \Omega$, and $\tilde{u}\in C^\infty(\Omega)$ satisfies 
$$\begin{cases}
L_\gamma(\tilde{u})=0\\

\tilde{u}_{|\partial \Omega} =u.
\end{cases}$$
with
$$L_\gamma(\tilde{u})={\rm div}(\gamma\nabla \tilde{u})=\sum_{i,j=1}^n\,\frac{\partial}{\partial x^i}\gamma^{ij}\frac{\partial \tilde{u}}{\partial x^j}$$

Lee and Uhlmann \cite{LU} observed that when $n\geq 3$, the voltage-to-current map can be expressed as the Dirichlet-to-Neumann operator $\mathcal{D}_{(\Omega,g)}$ with respect to a suitable Riemannian metric on $\Omega$.  Note that in arbitrary dimension, if $\gamma$ is the constant identity matrix, then the voltage-to-current map coincides with the Dirichlet-to-Neumann operator for the Euclidean metric on $\Omega$.   This is the case in Theorem~\ref{thm.cone} if one allows orbifolds.

\section{Asymptotics and invariants of the Steklov spectrum on orbisurfaces}\label{sec:asympt}

In this section, $\orb$ will always denote a compact orbisurface with boundary. The boundary of such an orbisurface $\orb$  consists of finitely many closed one-dimensional orbifolds.   As mentioned in Remark~\ref{typeI_II}, every closed one-dimensional orbifold is of one of two types\label{h-orbitype}:  a circle (type I) or the quotient of a circle by a reflection (type II).

\begin{defn}\label{data}\normalfont (i) Given a compact Riemannian orbisurface $(\orb,g)$ with boundary consisting of $r$ type I boundary components of lengths $\ell_1,\dots, \ell_r$ and $s$ type II boundary components of lengths $\bar{\ell_1},\dots, \bar{\ell_s}$, we will refer to the ordered pair $(L;\cb)$ of multisets $L=\{\lc_1,\dots, \lc_r\}$ and $\cb=\{\bar{\lc_1},\dots, \bar{\lc_s}\}$ as the \emph{boundary data} of $(\orb,g)$. 

(ii) Given $\ell_1,\dots, \ell_r, \bar{\ell_1},\dots, \bar{\ell_s}\in \R^+$, let $$S({\ell}_1,\dots, {\ell}_r;\bar{\ell_1},\dots, \bar{\ell_s}) =D{(\ell_1)\sqcup\cdots\sqcup D(\ell_r)\sqcup \ohd}(\bar{\ell_1})\sqcup\dots\sqcup\ohd(\bar{\ell_s})$$
where $D(\ell)$ and $\ohd(\bar{l})$ are defined as in Examples~\ref{spec.ball} and \ref{half_disk_orb}.
We will refer to $$S({\ell}_1,\dots, {\ell}_r;\bar{\ell_1},\dots, \bar{\ell_s})$$ as the \emph{canonical Riemannian orbisurface} with the given boundary data.
\end{defn}

In \S\S\ref{subsecasy}, we will show that the asymptotics of the Steklov spectrum of a Riemannian orbisurface $(\orb, g)$ are uniquely determined by the boundary data of $(\orb, g)$.  We will do this by showing that the Steklov spectrum of $(\orb,g)$ is asymptotic to that of the canonical Riemannian orbisurface with the same boundary data.  In \S\S\ref{invariants}, we will address the converse statement.

\subsection{Determining the Steklov asymptotics from the boundary data}\label{subsecasy}

\begin{thm}\label{thm.asympt}  In the notation of Definition~\ref{data}, let $(\orb,g)$ be a {compact Riemannian orbisurface} with boundary consisting of $r$ type I boundary components of lengths ${\ell}_1,\dots, {\ell}_r$ and $s$ type II boundary components of lengths ${\bar{\ell_1},\dots, \bar{\ell_s}}$. Then,
$$\Stek(\orb,g)\sim \Stek(S({\ell}_1,\dots, {\ell}_r;{\bar{\ell_1},\dots, \bar{\ell_s}}))$$
where for sequences $A=\{a_j\}$ and $B=\{b_j\}$, we write $A\sim B$ to mean $a_j-b_j=O(j^{-\infty}).$
\end{thm}

The proof of Theorem~\ref{thm.asympt} is similar to that in the manifold case \cite{GPPS} and proceeds through the steps outlined below.  The argument uses two crucial tools specific to dimension two, namely Lemma~\ref{lem.D} and the fact that conformally equivalent Riemannian metrics yield the same harmonic functions. 

\begin{itemize}
\item[Step 1.] We first prove the theorem in the special cases that $\orb$ is diffeomorphic to either a disk or to the half-disk orbifold $\ohd$. (The case of the disk is already carried out in \cite{GPPS}.)
\item[Step 2.] We observe that the boundary of $\orb$ is diffeomorphic to the boundary of the disjoint union $S$ of $r$ disks and $s$ half-disk orbifolds.   We construct a Riemannian metric $h$ on this disjoint union in such a way that a neighborhood of $\partial S$ in $(S,h)$ is isometric to a neighborhood of $\partial \orb$ in $(\orb,g)$.
\item[Step 3.]  We conclude by using Proposition~\ref{spectrumislocal} to see that $\Stek(\orb,g)\sim \Stek(S,h)$, and Step 1 to see that $\Stek(S,h)\sim \Stek(S({\ell}_1,\dots, {\ell}_r;{\bar{\ell_1},\dots, \bar{\ell_s}}))$.
\end{itemize}

In preparation for Step 1, we begin with the statement of uniformization for Riemannian metrics on $\ohd$.  We include the proof of the lemma here for completeness as we could not find a proof in the literature.

\begin{lemma}\label{lem.D} Let $D= \{(x,y) \in \R^2: x^2+y^2 \leq 1\}$ be a unit disk, $\tau_0$ the reflection of $D$ across the $x$-axis, and $g$ an arbitrary smooth $\tau_0$-invariant Riemannian metric on $D$.   Then there exists an isometry $F: (D,g)\to (\Om, e^\delta \geuc)$, where $(\Om,\geuc)$ is a simply-connected Euclidean domain and $\delta\in C^\infty(\Om)$.  Moreover, $\tau:= F\circ\tau_0\circ F^{-1}$ is an isometry (a reflection) of $(\Om, \geuc)$ leaving $\delta$ invariant.  Thus $\tau$ is also an isometry of $(\Om, e^\delta \geuc)$.
\end{lemma}

\begin{proof} 
\begin{comment}
Just as a Riemannian metric $g$ on a disk $D$ is flat if and only if $(D,g)$ is isometric to a simply-connected bounded plane domain $(\Om, g_{\euc})$, a metric $g$ on $\ohd$ is flat precisely when $(\ohd, g)$ is isometric to the quotient of a plane domain $(\Om, g_{\euc})$ by an isometric reflection.   Thus it is enough to prove the following statement:

\end{comment}

We will use a Ricci flow argument to prove this statement.  Let $D'=\{(x,y) \in \R^2: x^2+y^2 \leq 1+\epsilon \}$, and continue to denote by $\tau_0$ the reflection of $D'$ in the $x$-axis.   Here $\epsilon$ is chosen such that the metric $g$ on $D$ extends to a $\tau_0$-invariant Riemannian metric, still denoted $g$, on $D'$. Such an extension is always possible: we first extend $g$ arbitrarily to a Riemannian metric $h$ on $D'$, choosing $\epsilon$ so that $h$ is a valid metric, and then let $g=\frac{1}{2}(h +\tau_0^*h)$.  Double $D'$ across its boundary to obtain a topological sphere $S$ and a metric that is smooth except along the boundary edge $\mathcal{E}$ of $D'$.  We can now view $D\subset D'$ as a domain in $S$. The involution $\tau_0$ of $D'$ and its copy defines an involution $\tau_S$ of $S$.  Smooth out the metric in a neighborhood of $\mathcal{E}$ to obtain a $\tau_S$-invariant Riemannian metric $g_S$ on $S$ that agrees with $g$ on $D$.
 Under normalized Ricci flow, $g_S$ converges to a constant curvature metric $g_{cc}$ invariant under $\tau_S$ (see \cite[p. 105]{CK}).  Ricci flow in two dimensions is conformal, hence $g_S=e^{f_0} g_{cc}$ for some $f_0\in C^\infty(S)$.   Since both $g_S$ and $g_{cc}$ are $\tau_S$-invariant, so is $f_0$.      Thus there exists an isometry $H: (S,e^{f_0}g_{cc})\to (S, e^f g_{\can})$, where $g_{\can}$ is the canonical round metric and $f\in C^\infty(S)$, such that $H$ maps the fixed points of $\tau_S$ onto a great circle $\mathcal{C}$.   Hence R:= $H\circ\tau_S \circ H^{-1}$
is reflection across the great circle $\mathcal{C}$, and $f$ is invariant under $R$.  Perform stereographic projection $\operatorname{Ster}:S\to\R^2$ from a point on $\mathcal{C}$ lying outside of $H(D)$.  Then $\operatorname{Ster}$ carries $\mathcal{C}$ to a line $L$ and intertwines $R$ with reflection across $L$. Since $H(D)$ is invariant under $R$, it is carried to a domain $\Om$ symmetric about $L$.   By the conformality of stereographic projection, $(\operatorname{Ster}^{-1})^*(e^f g_{\can})=e^\delta\geuc$ for some $\delta\in C^\infty(\Om)$ invariant under reflection across $L$.   Define $F=\operatorname{Ster}\circ H$ to complete the proof.
\end{proof}

\begin{remark}\label{gohd} Since every Riemannian metric $g$ on the half-disk orbifold $\ohd=\langle \tau_0\rangle \backslash D$ (see Example~\ref{half_disk_orb}) pulls back to a $\tau_0$-invariant Riemannian metric on the disk $D$, we may interpret Lemma~\ref{lem.D} as saying that $(\ohd,g)$ is isometric to the conformally flat Riemannian orbifold $(\langle \tau_0\rangle \backslash\Omega, e^\delta\geuc)$.   

\end{remark}

\begin{proof}[Proof of Theorem~\ref{thm.asympt}]

\noindent{\bf Step 1.} We prove the theorem for $\orb=\ohd$.   The tools we will use are Lemmas~\ref{lem 2.1} and \ref{lem.D} and a result of Edward \cite[Prop. 1]{Ed}.  Edward proved that the full symbol of the Euclidean Dirichlet-to-Neumann operator on the boundary of a plane domain $\Om$ is given by $\operatorname{Symb}_{\euc} = \|\xi\|_{\euc}$ for $\xi\in T^*(\partial \Om)$.   More generally, if $h$ is a conformally flat metric on $\Om$, say
$h=e^{\delta} \geuc$, then we have $\operatorname{Symb}_h(x, \xi)=\|\xi\|_h$.
To see this, observe that the $h$-harmonic functions are exactly the $g_{\euc}$-harmonic functions, and the outward unit normals satisfy $\nu_h=\frac{1}{e^{\delta/2}}\nu_{\euc}$, implying that the Dirichlet-to-Neumann operators are related by
$\D_h=\frac{1}{e^{\delta/2}}\D_{\euc}.$  Hence $\operatorname{Symb}_h(x, \xi)=\frac{1}{e^{\delta(x)/2}}\|\xi\|_{\euc}=\|\xi\|_h$ as asserted.

We want to show that if $g$ is an arbitrary metric on $\ohd$, say with boundary length $\ell$, then the Steklov spectrum of $(\ohd,g)$ is asymptotic to that of the canonical flat half-disk orbisurface $\ohd(\ell)$.  By Lemma~\ref{lem.D} and Remark~\ref{gohd}, $(\ohd,g)$ is isometric to the quotient of $(\Om, h)$ by an isometric involution $\tau$, where $\Om$ is a simply-connected domain in $\R^2$ and $h$ is a conformally flat metric.  Also, $\ohd(\ell)$ is isometric to the quotient of $D(2\ell)$ by an isometric involution $\mu$.   The boundaries of both $(\Om, h)$ and $D(2\ell)$ are circles of length $2\ell$.  We may identify both boundaries with a fixed circle $S$ via length preserving diffeomorphisms in such a way that the involutions $\tau_{|\partial\Om}$ and $\mu_{|\partial D(2\ell)}$ correspond to the same involution $\rho$ of $S$.   Thus the Dirichlet-to-Neumann operators of each of $(\Om,h)$ and $D(2\ell)$ can be identified with pseudodifferential operators on $C^\infty(S)$ that are invariant under the $\Z_2$ action defined by $\rho$.  By the results quoted in the previous paragraph, the two operators have the same symbol, and thus differ by a smoothing operator.  The conclusion now follows from Lemma~\ref{lem 2.1} with $S$ playing the role of $M$ and $\Z_2$ the role of $G$.

\medskip

\noindent{\bf Step 2.}  The boundary of $\orb$ is diffeomorphic to the boundary of the disjoint union $S$ of $r$ disks and $s$ half-disk orbifolds. We now construct a Riemannian metric $h$ on this disjoint union in such a way that a neighborhood of $\partial S$ in $(S,h)$ is isometric to a neighborhood of $\partial \orb$ in $(\orb,g)$.
Any boundary component of type I has a collar neighborhood that is contained entirely in the set of regular points of $\orb$.   Thus if there are no type II boundary components, then one can proceed exactly as in \cite{GPPS}, capping off collar neighborhoods of each boundary component.  Hence we assume that $\orb$ has at least one boundary component of type II. We begin by doubling $\orb$ across all mirror reflectors to obtain a two-fold cover $\pi: \orbt \to \orb$. Continue to denote by $g$ the pullback to $\orbt$ of the metric $g$ on $\orb$. Note that each boundary component $B_i$ of $\orbt$ is necessarily of type I, and a sufficiently small collar neighborhood $U_i$ of each $B_i$ is contained in the set of regular points.  There exists a diffeomorphism $\varphi_i$ from $U_i$ to an annulus in $\R^2$. For each $B_i$ in $\orbt$ that is a double cover of a type II boundary component in $\orb$, note that $g_i=(\varphi_i^{-1})^*(g)$ is a metric on the annulus that is invariant with respect to the reflection symmetry $\tau_i$ induced from that on $U_i$.  As in the proof of \cite[Thm. 1.4]{GPPS}, we smoothly glue a disk to the non-$B_i$ boundary component of each annulus to obtain a topological disk $D_i$ with boundary $B_i$.  We extend the metric from our annuli to the $D_i$, noting that we may extend the invariant metrics in such a way that they remain reflection-invariant.  Then each $U_i$ is isometric to a small collar neighborhood of its corresponding $D_i$.  Moreover, for each $B_i$ in $\orbt$ that is a double cover of a type II boundary component in $\orb$, the map $\pi \circ \varphi_i^{-1}$ induces an isometry from a collar neighborhood of the boundary of the Riemannian orbisurface $(\langle \tau_i \rangle \backslash D_i, g_i)$ to a collar neighborhood of the corresponding type II boundary component in $\orb$, where we are denoting by $g_i$ the metric on $\langle \tau_i \rangle \backslash D_i$ induced by the metric of the same name on $D_i$. This gives us a metric $h$ on $S$ with the desired property.

 The theorem now follows as in Step 3 explained above.
 \end{proof}
 
\subsection{Steklov spectral invariants}\label{invariants}
\begin{defn}\label{def.seq}\normalfont Given $l\in\R^+$, let $A(\ell)$ be the multiset
$$A(\ell)=\left\{0, \frac{2\pi}{\ell},  \frac{2\pi}{\ell}, \frac{4\pi}{\ell},  \frac{4\pi}{\ell},  \frac{6\pi}{\ell}, \frac{6\pi}{\ell}{,\cdots}\right\}\;$$
that is, $A(\ell)$ consists of $0$ together with two copies of $\frac{2\pi}{\ell}\N$.  Let
$$\ab(\ell)=\{0\}\cup \frac{2\pi}{\ell}\N.$$
Given finite multisets $L=\{\lc_1,\dots, \lc_r\}$ and $\cb=\{\bar{\lc_1},\dots, \bar{\lc_s}\}$ of elements of $\R^+$,  let $\sigma(L;\cb)$ be the monotone non-decreasing sequence consisting of all the elements (repeated with multiplicities) of $A(\lc_1)\sqcup A(\lc_2)\sqcup\dots\sqcup A(\lc_r)\sqcup \ab(2\bar{\lc_1})\sqcup\dots\sqcup \ab(2\bar{\lc_s})$.  Write $\sigma_j(L;\cb)$ for the $j$th element of this sequence.

\end{defn}

By Examples~\ref{spec.ball} and \ref{half_disk_orb}, we see that
$$\Stek(S({\ell}_1,\dots, {\ell}_r;{\bar{\ell_1},\dots, \bar{\ell_s}}))=\sigma(L;\cb).$$

\begin{cor}\label{precise} Under the hypotheses of Theorem~\ref{thm.asympt}, the Steklov eigenvalues of $(\orb,g)$ satisfy
$$\sigma_j(\orb,g)=\sigma_j(L;\cb) +O(j^{-\infty}).$$

\end{cor}

Corollary~\ref{precise} gives us precise asymptotics of the Steklov spectrum of a compact Riemannian orbisurface in terms of its boundary data (Definition~\ref{data}). We now investigate the inverse problem and ask how much information about the boundary data of a compact Riemannian orbisurface can be gleaned from its Steklov spectrum.  As Example ~\ref{ell1ell2} shows, we have $\sigma(\{\ell_1\};\{\ell_2/2,\ell_2/2\})=\sigma(\{\ell_{2}\};\{\ell_1/2,\ell_1/2\})$; certain interchanges between different types of boundary components do not affect the asymptotics of the Steklov spectrum.  Thus we begin by encoding this potential Steklov isospectrality as an equivalence relation on the set of possible boundary data.

\begin{defn}\label{defn.equivclass}\normalfont
Let $L_i=\{\ell_1^{(i)},\ldots,\ell_{r_i}^{(i)}\}$ and $\overline{L}_i=\{\bar{\lc}_1^{(i)},\ldots,\bar{\lc}_{s_i}^{(i)}\}$, $i=1,2$, be multisets of positive real numbers. We define an equivalence relation by $(L_1; \overline{L}_1)\equiv(L_2; \overline{L}_2)$ if and only if the following equalities hold:
\begin{itemize}
\item[(i)]$r_1=r_2$;
\item[(ii)] $s_1=s_2$;
\item[(iii)] $L_1\sqcup L_1\sqcup 2\overline{L}_1=L_2\sqcup L_2\sqcup 2\overline{L}_2,$
where $2\overline{L}$ denotes the multiset consisting of each element of $\overline{L}$ multiplied by $2$.  The equality should be understood as equality of multisets (i.e., multiplicities are included).
\end{itemize}
Note that although conditions (i) and (iii) imply condition (ii), we include it for clarity.
\end{defn}

\begin{remark}\label{singleton}\normalfont Observe that the equivalence class of $(L; \overline{L})$ consists of a single element if either of $L$ or $\overline{L}$ is empty, or if all elements of $\overline{L}$ have multiplicity one.
\end{remark}

\begin{lemma}\label{lem.equivtogamma}
In the notation of Definitions \ref{def.seq} and \ref{defn.equivclass}, the following are equivalent:
\begin{itemize}
\item[(a)] $\sigma(L_1; \overline{L}_1)=\sigma(L_2; \overline{L}_2)$;
\item[(b)]  $\sigma_j(L_1; \overline{L}_1)-\sigma_j(L_2; \overline{L}_2) = O(j^{-\infty});$
\item[(c)] $(L_1; \overline{L}_1)\equiv(L_2; \overline{L}_2)$.
\end{itemize}
\end{lemma}

\begin{proof} The implication (a)$\implies$(b) is trivial, and (c)$\implies$(a) is an easy consequence of Definitions~ \ref{def.seq} and \ref{defn.equivclass}.  It remains to prove (b)$\implies$(c).

We first show that (b) implies condition (iii) in Definition~\ref{defn.equivclass}.  Each of the sequences $\sigma(L_i; \overline{L}_i)$, $i=1,2$, can also be viewed as a multiset by ignoring the ordering of the elements while retaining their multiplicities.  Both multisets consist of finitely many zeroes together with a disjoint union of arithmetic progressions of the form $\alpha\N$ with $\alpha>0$.  Condition (iii) of Definition~\ref{defn.equivclass} says precisely that the same arithmetic progressions occur in both multisets and with the same multiplicities.   We will apply a result of A.~Girouard, L.~Parnovski, I.~Polterovich and D.~Sher \cite{GPPS} for comparing multisets that are unions of arithmetic progressions.   We first recall some definitions. Let $A$ and $B$ be two multisets of positive real numbers. A map $\Phi:A\to B$ is \textit{close} if for every $\epsilon> 0$, there are only finitely many $x \in A$ with
$|\Phi(x) - x|\ge\epsilon$. A map $\Phi:A\to B$ is  an \textit{almost-bijection} if for all but finitely many $y\in B$ the pre-image $\Phi^{-1}(y)$ consists of one point.

Condition (b) implies the existence of a close almost-bijection between the multisets $\sigma(L_1; \overline{L}_1)$ and $\sigma(L_2; \overline{L}_2)$.  We may then apply \cite[Rem. 2.10]{GPPS} to conclude that the same arithmetic progressions occur and with the same multiplicities.  Thus condition (iii) of Definition~\ref{defn.equivclass} holds.

It remains to verify conditions (i) and (ii) of Definition~\ref{defn.equivclass}.  Condition (iii), which we have already verified, tells us that $2r_1+s_1 = 2r_2 + s_2$ and that the two multisets $\sigma(L_1; \overline{L}_1)$ and $\sigma(L_2; \overline{L}_2)$ are identical except possibly for the number of zeros ($r_1+s_1$ versus $r_2+s_2$).  Suppose they are not equal, e.g., suppose that $\sigma(L_1; \overline{L}_1)$ has $k$ more zeros than $\sigma(L_2; \overline{L}_2)$.  Then by the hypothesis (b), we get
$$\sigma_{j+k}(L_1; \overline{L}_1)=\sigma_j(L_2; \overline{L}_2)=\sigma_j(L_1; \overline{L}_1)+O(j^{-\infty}).$$
Thus we have
$$\sigma_{j+1}(L_1; \overline{L}_1)-\sigma_j(L_1; \overline{L}_1)\leq \sigma_{j+k}(L_1; \overline{L}_1)-\sigma_j(L_1; \overline{L}_1)=O(j^{-\infty}).$$
This is a contradiction by using the fact that
$\lim \sup\ \sigma_{j+1}(L_1;\overline{L}_1)-\sigma_j(L_1;\overline{L}_1)=\frac{2\pi}{\ell}$ where $\ell$ is the largest element of
$L_1\sqcup2\overline{L}_1$; the proof of this fact is given in the proof of \cite[Thm. 1.7]{GPPS}. Hence, $r_1+s_1 = r_2+s_2$, and since $2r_1+s_1 = 2r_2 + s_2$, we see that conditions (i) and (ii) of Definition~\ref{defn.equivclass} hold. Thus $(L_1; \overline{L}_1)\equiv(L_2; \overline{L}_2)$.
\end{proof}

For a compact Riemannian orbisurface $(\orb, g)$, we know from Theorem~\ref{thm.asympt} that the Steklov spectrum of $\orb$ is asymptotic to $\sigma(L; \overline{L})$ for some multisets of positive real numbers $L$ and $\overline{L}$. Lemma~\ref{lem.equivtogamma} implies that $(L; \overline{L})$ is unique up to equivalence, proving the following theorem.

\begin{thm}\label{h-audiblebd}
In the language of Definition~\ref{data}, the Steklov spectrum of a compact Riemannian orbisurface $(\orb, g)$ determines the equivalence class of the boundary data $(L;\cb)$.
\end{thm}

Theorem \ref{h-audiblebd} yields two corollaries.

\begin{cor}\label{cor.SpecDeterminesNumber}
Let $(\orb, g)$ be a compact Riemannian orbisurface.  The Steklov spectrum uniquely determines the number of type I and type II boundary components, respectively.  In particular, the Steklov spectrum detects the presence of singular points on the boundary of $\orb$, as well as the number of singular points.
\end{cor}

\begin{cor}\label{cor.boundary}
Let $(\orb, g)$ and $(\orb',g')$ be Steklov isospectral compact Riemannian orbisurfaces with boundary.   If either
\begin{itemize}
\item[(a)] all boundary components of $\orb$ have the same type,
\item[(b)] or if all boundary components of $(\orb,g)$ of type II have different lengths,
\end{itemize}
then $(\orb,g)$ and $(\orb',g')$ must have the same boundary data $(L;\cb)$. Thus, for generic compact Riemannian orbisurfaces with boundary, the Steklov spectrum determines the boundary data.

\end{cor}

\begin{proof} The corollary is immediate from Theorem~\ref{h-audiblebd} and Remark~\ref{singleton}.
\end{proof}
%%%%%%%%%%%%%%%%%%%%%%%%%%%%%%%%%%%%%%%%%%%%%%%%%%%%%%%%%%%%%%%%%%%%%%%%%%%%%%%%%%%%%%%%%%%%%%%%%%%%%%%%%%%%%%%%%%%%%%%%%%%%%%%%%%%%%%%%%%%%%%%%%%%%%%

\section{Examples of Steklov Isospectral Orbifolds}\label{sec:examples}

In this section we discuss constructions of Steklov isospectral orbifolds.  

There is a large literature on constructions of \emph{Laplace} isospectral compact Riemannian manifolds and orbifolds, with or without boundary.    The known techniques for constructing examples essentially fall into three types:

\begin{itemize}
\item representation theoretic methods such as Sunada's Theorem \cite{sun} and its generalizations (see the survey \cite{Gsurvey});

\item the method of torus actions (see, for example, \cite{G2001,GSz,GW,Schueth01}); and

\item methods specific to special Riemannian manifolds such as flat closed manifolds (e.g., \cite{DMM92,MR09}) or Lens spaces (e.g., \cite{Ike80,LMR,sh11}) in which the spectrum can be ``explicitly'' computed, e.g., through the use of a generating function.
\end{itemize}

In \cite{GHW}, P. Herbrich, D. Webb and the third author showed that both the original Sunada technique and the torus action method, when applied to compact Riemannian manifolds with boundary, result in manifolds that are Steklov isospectral as well as Laplace isospectral.   These methods are valid for orbifolds as well as manifolds; in particular, the many examples of Laplace isospectral manifolds and orbifolds in the literature constructed by these methods are also Steklov isospectral.   We will not repeat these examples here and will instead focus on the following:
\footnote{We mention here one result of \cite{GHW} that uses orbifolds in a crucial way.  The Laplace isospectrality of planar domains such as those constructed in \cite{GWW} does not follow immediately from Sunada's Theorem. Instead, Sunada's Theorem yields pairs of orbifolds whose underlying spaces are plane domains, and an argument specific to the Laplace spectrum shows that the underlying plane domains are Laplace isospectral. However, as shown in  \cite{GHW}, one can conclude Steklov isospectrality of the orbifolds, but not of the underlying plane domains.  See \cite{GHW} for more details.}

\begin{itemize}
\item ad hoc methods for constructing Steklov isospectral orbifold quotients of Euclidean balls; 

\item examples illustrating the failure in dimensions $n>2$ of the Steklov spectrum to detect how many of the boundary components of an orbifold contain singularities, in contrast to Corollary~\ref{cor.SpecDeterminesNumber};

\item a construction of families of Steklov isospectral bad orbifolds, which we obtain by adapting a construction of M. Weilandt \cite{Weil}.   (Weilandt constructed Laplace isospectral bad closed orbifolds using the torus action method.)
\end{itemize}

\subsection{Orbifold quotients of Euclidean balls}\label{quo_balls}

We have already computed the Steklov spectrum of a Euclidean ball $B(0,R)$ in Example~\ref{spec.ball}.  If $\Gamma$ is a finite subgroup of $O(n)$, then the eigenfunctions of the quotient orbifold $\orb=\Gamma\bs B(0,R)$ pull back to the $\Gamma$-invariant homogeneous harmonic polynomials on $B(0,R)$.

\begin{ex}\label{spherical} Given finite subgroups $\Gamma_1$ and $\Gamma_2$ of $O(n)$, the quotients $\orb_1=\Gamma_1\bs B(0,R)$ and $\orb_2=\Gamma_2\bs B(0,R)$, with radius $R$ arbitrary but fixed, are Steklov isospectral if and only if the spherical space forms $\Gamma_1\bs S^{n-1}$ and $\Gamma_2\bs S^{n-1}$ are Laplace isospectral.   Indeed, letting $d_{\Gamma_i}(m)$, i=1,2, denote the dimension of the space of $\Gamma_i$-invariant homogeneous harmonic polynomials of degree $m$, Example~\ref{spec.ball} shows that $\orb_1$ and $\orb_2$ are Steklov isospectral if and only if $d_{\Gamma_1}(m)=d_{\Gamma_2}(m)$ for all $m$, which is exactly the condition for the corresponding spherical space forms to be isospectral.   We note that there is a large literature on Laplace isospectral spherical space forms. Here we are using the term \emph{spherical space form} to denote any quotient of the unit sphere $S^{n-1}$ by a finite subgroup of $O(n)$.  Thus the class of spherical space forms includes many orbifolds as well as manifolds. 

\end{ex}

As already discussed in Section~\ref{sec:invtomog}, we obtain a more interesting example in dimension two by considering quotients of disks of different radii by cyclic groups.   There the quotients had not only the same Steklov spectrum but also the same Dirichlet-to-Neumann maps.

\subsection{Isospectral orbifolds with multiple components}

O. Parzanchevski \cite{Parz} generalized Sunada's technique to address Laplace isospectrality of manifolds and orbifolds with multiple components.    The following proposition asserts that his technique, when applied to manifolds or orbifolds with boundary, results in Steklov isospectrality as well.

\begin{prop}\label{parzprop}
Let $G$ be a finite group and $\{H_1,\dots,H_r\}$ and $\{K_1,\dots,K_r\}$ be two collections of subgroups of $G$. Suppose for each $x$ in $G$, 
\begin{equation}\label{eqn.sun}\sum_{i=1}^{r} \frac{|[x]\cap H_i|}{|H_i|} = \sum_{i=1}^{r} \frac{|[x]\cap K_i|}{|K_i|}{,} \end{equation} 
where $[x]$ denotes the conjugacy class of $x$ in $G$. Let $(M,g)$ be a compact Riemannian manifold (or orbifold) with boundary such that $G$ acts isometrically on $M$. Then $$Stek \left(\sqcup_{i=1}^r H_i \backslash M \right)=Stek \left(\sqcup_{i=1}^r K_i \backslash M \right).$$
\end{prop}

Note that because $G$ is not required to act freely on $M$, the quotients $H_i \backslash M$ and $K_i \backslash M$ may be orbifolds even if $M$ is a smooth manifold.  When $r=1$, the proposition reduces to Sunada's Theorem \cite{sun}.

\begin{proof}
The proof is essentially the same as that of Parzanchevski's result.   Let $X=\cup_{i=1}^r\,H_i\bs G$ and $Y=\cup_{i=1}^r\,K_i\bs G$, and let $\C[X]$ and $\C[Y]$ denote the vector spaces consisting of all formal linear combinations of elements of  $X$ and $Y$, respectively.  Write $\orb_H= \sqcup_{i=1}^r H_i \backslash M$ and $\orb_K=\sqcup_{i=1}^r K_i \backslash M$.  The action of $G$ by translation on the various coset spaces gives rise to a linear action of $G$ on each of $\C[X]$ and $\C[Y]$, permuting the basis elements.   Parzanchevski first shows that Equation~(\ref{eqn.sun}) holds if and only if the representations of $G$ on $\C[X]$ and $\C[Y]$ are linearly equivalent.   Generalizing the so-called transplantation proof of Sunada's Theorem, he then shows that this linear equivalence gives rise to an explicit ``transplantation'' map $\tau: L^2(\orb_H)\to L^2(\orb_K)$.  The transplantation carries smooth functions to smooth functions and intertwines the Laplacians. In particular, it carries harmonic functions to harmonic functions. Since $G$ also acts isometrically on the boundary of $M$, the linear equivalence of $\C[X]$ and $\C[Y]$ further gives rise to a transplantation map  $\tau_\partial: C^\infty(\partial \orb_H)\to C^\infty(\partial \orb_K)$.  The following diagram commutes, where the downward arrows are the restriction maps, i.e., $r_H(f)=f_{|\partial \orb_H}$ and similarly for $r_K$.

\[
\begin{array}{ccc}
C^\infty(\orb_H) & \xrightarrow{\tau} & C^\infty(\orb_K)\\
\downarrow{r_H} & & \downarrow{r_K}\\
C^\infty(\partial \orb_H) & \xrightarrow{\tau_\partial} & C^\infty(\partial \orb_K).
\end{array}
\] 
 Finally, the transplantations intertwine the normal derivatives across the boundaries: $\tau_\partial (\partial_\nu f)=\partial_\nu(\tau(f))$ for $f\in C^\infty(\orb_H)$.   
 
 It follows that $\tau_\partial$ intertwines the Dirichlet-to-Neumann maps: $\tau_\partial\circ\mathcal{D}_{\orb_K}=\mathcal{D}_{\orb_H}\circ\tau_\partial$.   The proposition follows.
\end{proof}

\begin{ex}\label{ex.quot1} Let $G=\{1,\sigma,\tau,\sigma\tau\}$ be the Klein 4-group, with the two collections of subgroups $H_1=\{1,\sigma\},\,H_2=\{1,\tau\},\,H_3=\{1,\sigma\tau\}$, and $K_1=\{1\},\,K_2=G,\,K_3=G$.   Note that these subgroups satisfy Equation~(\ref{eqn.sun}).  (The group $G$ and these collections of subgroups were also used in \cite{Parz}.)

 Define an action of $G$ on the Euclidean 3-ball $M:=B(0,1)\subset \R^3$ by letting $\sigma$, $\tau$ and $\sigma\tau$ act by rotation through angle $\pi$ about the $x$, $y$ and $z$-axes, respectively. By Proposition~\ref{parzprop}, $\orb_H:=\sqcup_{i=1}^3 \left(H_i \backslash M \right)$ is Steklov isospectral to $\orb_K:=\sqcup_{i=1}^3 \left(K_i \backslash M \right)$.
\end{ex}
Observe that all three components of $\orb_H$ have singularities both in their interiors and on their boundaries.  Each of their boundaries is a (2,2)-football.   On the other hand, one of the components of $\orb_K$ is a smooth manifold (the ball $B(0,1)$ with boundary a sphere) while the boundaries of the other two components are (2,2,2)-pillows.

\begin{ex}\label{ex.quot2} We again let $G$ be the Klein 4-group and use the same collection of subgroups as in the previous example.  Define a different isometric action of $G$ on $M:=B(0,1)$ by letting $\sigma$, $\tau$ and $\sigma\tau$ act as reflections across the $xy$-plane, the $xz$-plane, and the $yz$-plane, respectively. We see that the orbifolds $\orb_H:=\sqcup_{i=1}^3 \left(H_i \backslash M \right)$ and $\orb_K:=\sqcup_{i=1}^3 \left(K_i \backslash M \right)$ are Steklov isospectral by Proposition~\ref{parzprop}.
\end{ex}

This example illustrates properties similar to those in the previous example and also illustrates that the Steklov spectrum does not determine how many of the components are orientable.   The smooth component of $\orb_K$ is orientable while all components of $\orb_H$, as well as their boundaries, are non-orientable.

Examples~\ref{ex.quot1} and \ref{ex.quot2} yield that Corollary \ref{cor.SpecDeterminesNumber} fails in higher dimensions:

\begin{prop}\label{cor.fails}     The Steklov spectrum does not determine the number of smooth versus singular boundary components in dimensions greater than two. 

\end{prop}

\subsection{Steklov isospectral bad orbifolds}
In \cite{Weil}, M. Weilandt constructed continuous families of Laplace isospectral metrics on a bad closed orbifold.  These were the first examples of isospectral metrics on bad orbifolds.   In the following example, we modify his construction to obtain families of non-isometric Steklov isospectral bad orbifolds. 

 Given relatively prime positive integers $p$ and $q$, define a smooth action of the circle $S^1$ on $\C^{n+1}$ by 
$$\sigma(u,v)=(\sigma^pu,\sigma^q v)$$
for $\sigma\in S^1,\,u\in \C^{n-1},\, v\in \C^2$.   As pointed out in \cite{Weil}, this action is effective and the isotropy group $I_{(u,v)}$ at $(u,v)$ is given as follows:   If both $u$ and $v$ are non-zero, then $I_{(u,v)}$ is trivial.   For $u$, respectively $v$, non-zero, $I_{(u,v)}$ consists of the $p$-th, respectively $q$-th, roots of unity.   Of course, $I_{(0,0)}=S^1$.   In particular, the action on $\C^{n+1}\setminus\{0\}$ has only finite isotropy.   Weilandt considered the $2n$-dimensional weighted projective space $\orb(p,q)$ given by the quotient of the unit sphere $S^{2n+1}\subset\C^{n+1}$ by this $S^1$ action.   As pointed out in \cite{Weil}, $\orb(p,q)$ is a simply-connected bad orbifold when $n\geq 4$ except in the case $p=q=1$.  

\begin{prop}\label{prop.an}
For $0<r<R$ and $n\ge4$, let $A(0,r,R)$ be the annulus in $\C^{n+1}$ of inner radius $r$ and outer radius $R$ relative to the Euclidean metric, and let $\orb(p,q;r,R)$ denote the quotient of $A(0,r,R)$ by the action of $S^1$ defined above. Then there exists a family $g_t$ of Steklov isospectral, non-isometric, Riemannian metrics on the bad orbifold $\orb(p,q;r,R)$.
\end{prop}
\begin{proof}[Proof]
Note that $\orb(p,q;r,R)$ is a bad orbifold since each of its two boundary components is diffeomorphic to the bad orbifold $\orb(p,q)$.

To show the existence of such a family of metrics $g_t$, we modify Weilandt's construction, which is based on an earlier construction of D. Schueth \cite{Schueth01}.  Schueth used the method of torus actions to construct families of Riemannian metrics $\tilde{g}_t$ on $\C^{n+1}$, $n\geq 3$, such that the restrictions of these metrics to any ball centered at the origin -- or more generally, the restrictions to any radially symmetric compact domain in $\C^{n+1}$ -- are isospectral with respect to both Dirichlet and Neumann boundary conditions.  Moreover, she showed that the metrics, which we continue to denote by $\tilde{g}_t$, induced on any sphere centered at the origin are also Laplace isospectral.  

 Let $n\geq 4$, let  $(p,q)$ be a pair of relatively prime integers, and consider the $S^1$ action on $\C^{n+1}$ defined above.  Schueth's metrics $\tilde{g}_t$ are invariant under this $S^1$ action and thus induce a family $\{g_t\}$ of Riemannian metrics on the quotient $S^1\bs (\C^{n+1}\setminus\{0\})$ by this action.   We claim that the restrictions of these metrics to 
$\orb(p,q;r,R)$ are Steklov isospectral (as well as Laplace isospectral with respect to Dirichlet and Neumann boundary conditions).  
 
Weilandt used the torus action method, adapted to the orbifold setting, to show that the induced metrics on $\orb(p,q)=S^1\bs S^{2n+1}$ are Laplace isospectral.   One can imitate Weilandt's argument to see that the metrics $g_t$ on $\orb(p,q;r,R)$ also satisfy the hypotheses of the torus action method, as adapted to the Steklov spectrum in \cite{GHW}, thus proving the proposition.
\end{proof}

%%%%%%%%%%%%%%%%%%%%%%%%%%%%%%%%%%%%%%%%%%%%%%%%%%%%%%%%%%%%%%%%%%%%%%%%%%%%%%%%%%%%%%%%%%%%%%%%%%%%%%%%%%%%%%%%%%%%%%%%%%%%%%%%%%%%%%%%%%%%%%%%%%%%%%
\section{Upper bounds on Steklov eigenvalues}\label{sec:bounds}

In this section we examine various upper bounds on the Steklov eigenvalues of a Riemannian orbifold scaled by a power of the volume of its boundary, particularly noting how  these upper bounds differ somewhat from those familiar from the manifold setting due to the presence of orbifold structure.  

The upper bounds obtained in this section are all built on Theorem \ref{h-orbibd} below, which generalizes Theorem 4.1 in \cite{H}.   Theorem \ref{h-orbibd} requires that the orbifold $(\orb,g)$ be conformally embedded as a domain in a complete orbifold with Ricci curvature bounded below.  In Theorem~\ref{h-confinv}, we will remove this assumption using a conformal invariant that allows us to express the upper bound in terms of the intrinsic geometry of $(\orb,g)$.  Then in Theorem~\ref{h-2dupbd}, we will obtain a purely topological upper bound in dimension two. Finally in \S\S\ref{laplace} we observe that upper bounds on Neumann Laplace eigenvalues analogous to those of Theorems \ref{h-orbibd}, \ref{h-confinv} and \ref{h-2dupbd} can be obtained using similar methods.

Given an orbifold $\orb$ and Riemannian metric $g$ on $\orb$, we denote by $[g]$ the conformal class of $g$ as in Definition~\ref{def:conf_class}.

\begin{thm}\label{h-orbibd} Let $(\OP,h_0)$ be a complete $n$-dimensional Riemannian orbifold  with $\ric_{(\OP,h_0)}\ge-\alpha(n-1)$, $\alpha\ge0$, and $\orb$ a bounded domain in $\OP$ (both $\orb$ and $\partial\orb$ may contain singular points). Then for every metric $g\in[h_0|_\orb]$ on $\orb$ and every $k\in\N$,
\begin{equation}\label{h-orbibd1}
\sigma_k(\orb,g)\vol_g(\partial\orb)^{\frac{1}{n-1}}\le \frac{C_1\alpha\vol_{h_0}(\orb)^{\frac{2}{n}}+C_2k^{\frac{2}{n}}}{\iso_g(\orb)^{1-\frac{1}{n-1}}},
\end{equation}  
where $C_1$ and $C_2$ are positive constants depending only on $n$, and $\iso_g(\orb)$ is the isoperimetric ratio of $\orb$ given by
\[
\iso_g(\orb)=\frac{\vol_g(\partial \orb)}{\vol_g(\orb)^\frac{n-1}{n}}.
\]
In even dimensions $n=2m>2$, the power of $\iso_g(\orb)$ in the denominator is sharp from below; i.e., $1 - \frac{1}{n-1}$ cannot be replaced by any smaller power.
\end{thm}

\begin{remark}\label{rem.sharp}\normalfont 
(i) Since extreme cases of $\iso_g(\orb)\gg1$ and $\iso_g(\orb)\ll1$ can happen in both the manifold and more general orbifold setting,  it is of interest to determine the sharpness of the power of the isoperimetric ratio in Equation~\eqref{h-orbibd1} from both above and below.\\
(ii) Note that the denominator on the right side in \eqref{h-orbibd1} is trivial when $n=2$.  In light of Weyl's asymptotic formula, one can hope to improve the power of $k$ in the numerator to $\frac{1}{n-1}$ when $n>2$. 
\end{remark}
Although Equation~(\ref{h-orbibd1}) is identical to that obtained in the manifold setting in \cite{H}, to our knowledge the sharpness statement was neither evident nor known in the manifold setting. 

The sharpness statement will follow from Example~\ref{h-ex} below.  The proof of Equation~(\ref{h-orbibd1}) parallels that given in \cite[Thm. 4.1]{H} in the manifold setting.  For the convenience of the reader, in \S\S\ref{sect:h-orbibd_pf} we summarize the main ideas of the proof of  Equation~(\ref{h-orbibd1}), indicating any adaptations necessary in the orbifold setting.

We now define the conformal invariant that will allow us to reformulate the inequality \eqref{h-orbibd1} in terms of the intrinsic geometry of $(\orb,g)$.  
 
\begin{defn}\label{def:confinv}\normalfont 
Given an $n$-dimensional compact Riemannian orbifold $(\orb,g)$ with boundary, we say $(\OP,h)$ \emph{is an admissible extension of} $(\orb,g)$ if the following conditions hold:
\begin{itemize}
\item[(i)] $(\OP,h)$ is a complete, $n$-dimensional Riemannian orbifold;
\item[(ii)] $(\OP,h)$ has Ricci curvature bounded below; and
\item[(iii)] $(\orb,g)$ conformally embeds as a bounded subdomain of $(\OP,h)$.
\end{itemize}
 
We define a conformal invariant $\conf(\orb,g)$ of $(\orb,g)$ as 
\[
\begin{split}
\conf(\orb,g):=\inf\{\beta\ge0:&~\text{there exists an admissible extension}~(\OP,h)\\
&~\text{of}~(\orb,g)~\text{with}~\vol_{h}(\orb)^{\frac{2}{n}}\ric_{(\OP,h)}\ge-\beta(n-1)\}.
\end{split}
\]
\end{defn}

\begin{remark}
Admissible extensions always exist.  Here is one way to obtain such an extension:   An orbifold collar theorem (see \cite{Dr94}, page 304) states that there is a neighborhood of $\partial\orb$ in $\orb$ that is orbifold diffeomorphic to $\partial\orb\times [0,1]$.   By extending the collar beyond the boundary, we obtain a new compact orbifold $\orb'$ with boundary containing $\orb$ in its interior.  Again following \cite{Dr94}, we can double $\orb'$ across its boundary to obtain a closed orbifold $\orb''$ containing $\orb$.  We then smoothly extend the Riemannian metric on $\orb$ to a Riemannian metric on $\orb''$.    Since $\orb''$ is compact, the Ricci curvature is bounded below, so this construction yields an admissible extension.  (For an alternative construction in dimension two, see Lemma~\ref{h-orbiembed}.)
\end{remark}

\begin{remark}\label{epsilontrick}\normalfont 
Suppose $(\orb,g)$ has an admissible extension with nonnegative Ricci curvature. Then it is immediate that $\conf(\orb,g)=0$.
\end{remark}

Since the left side of \eqref{h-orbibd1} does not depend on the choice of admissible extension $(\OP,h_0)$ of $(\orb,g)$, we obtain the following consequence of Theorem~\ref{h-orbibd} and Definition~\ref{def:confinv}.
  
\begin{thm}\label{h-confinv}
 Let $(\orb,g)$ be an $n$-dimensional compact Riemannian orbifold with boundary. In the notation of Definition~\ref{def:confinv} we have, for every $k\in\N$,  
\begin{equation}\label{h-confinv1}
\sigma_k(\orb,g)\vol_g(\partial\orb)^{\frac{1}{n-1}}\le \frac{C_1\conf(\orb,g)+C_2k^{\frac{2}{n}}}{\iso_g(\orb)^{1-\frac{1}{n-1}}}.
\end{equation}
In even dimensions $n=2m>2$, the power of $\iso_g(\orb)$ in the denominator is sharp from below. 
\end{thm}

We now prove the sharpness statement concerning the isoperimetric constant in Theorems~\ref{h-orbibd} and \ref{h-confinv}.

\begin{lem}\label{lem.q} Let $(\Omega,g)$ be a compact $n$-dimensional Riemannian manifold with boundary, let $\Gamma$ be a group of finite order $q$ that acts effectively by isometries on $\Omega$, and let $\orb=\Gamma\bs \Omega$.   Continue to denote by $g$ the induced Riemannian metric on $\orb$.  Then 
$$\iso_g(\orb)=q^{-\frac{1}{n}}\iso_g(\Omega).$$

\end{lem}

\begin{proof}The lemma is immediate from the fact that $\vol_g(\orb)=\frac{1}{q}\vol_g(\Omega)$ and $\vol_g(\partial\orb)=\frac{1}{q}\vol_g(\partial \Omega)$.
\end{proof}

 \begin{ex}\label{h-ex}
 Let $n=2m$ be even, let $B(0,1)$ be the Euclidean unit ball in $\R^{n}$ centered at 0, and let $S^{n-1}$ be its boundary sphere.   Given $q\in \Z^+$ and $\pv=(p_1,p_2,\ldots,p_{m})\in \Z^{m}$, 
  let $\Gamma_{q,\pv}$ be the cyclic subgroup of $O(n)$ generated by the orthogonal transformation
$$\gamma_{q,\pv}(z_1,z_2,\ldots,z_m)=(e^{2\pi i \frac{p_1}{q}}z_1,e^{2\pi i \frac{p_2}{q}}z_2,\ldots, e^{2\pi i \frac{p_{m}}{q}}z_m)$$ for all $(z_1,z_2,\ldots z_m)\in  \C^m\simeq \R^{n}$.   Set
$$\orb(q;\pv)=\Gamma_{q,\pv}\bs B(0,1)\mbox{\,\,and\,\,}L(q;\pv)=\Gamma_{q,\pv}\bs S^{n-1}$$
with the induced Euclidean, respectively round, Riemannian metric.   Then $\orb(q;\pv)$ is necessarily an orbifold with singularities; the boundary $L(q;\pv)$ of $\orb(q;\pv)$ is a lens space which may either be a smooth manifold or an orbifold with singularities depending on whether the $p_i$'s are relatively prime to $q$.  

To apply Theorem~\ref{h-orbibd}, we take $\mathcal{P}=\Gamma_{q,\pv}\bs\R^n$ and $\alpha=0$.  In Theorem~\ref{h-confinv}, we have $\conf(\orb(q;\pv))=0$ by Remark~\ref{epsilontrick}.  Since $\vol(\partial \orb)=\frac{1}{q}\vol(S^{n-1})$, both theorems say that there exists a constant $C_2$ such that
$$\sigma_k(\orb(q;\pv))\vol(S^{n-1})^{\frac{1}{n-1}}q^{-\frac{1}{n-1}}\leq C_2\iso(B(0,1))^{\left(\frac{1}{n-1}-1\right)}k^{\frac{2}{n}}q^{\frac{1}{n}\left(1-\frac{1}{n-1}\right)}.$$
Thus for each fixed $k$ there exists a constant $C$ such that
\begin{equation}\label{sig2}\sigma_k(\orb(q;\pv))\leq Cq^{\frac{2}{n}}=Cq^{\frac{1}{m}}\end{equation}
for all $q\in \Z^+$ and $\pv\in \Z^m$.
(Here the constant $C$ involves only $C_2k^{\frac{2}{n}}$, $\vol(S^{n-1})$ and $\iso(B(0,1))$.)

If the power of $\iso_g(\orb)$ appearing in the denominator of the bound in the two theorems could be lowered, then we would have a corresponding smaller power of $q$ appearing in the right side of Equation~(\ref{sig2}).   We will now show this is impossible.

As noted at the start of \S\S \ref{quo_balls}, $\sigma_2(\orb(q;\pv))$ is the minimum $r\in\Z^+$ for which there exists a $\gqp$-invariant homogeneous harmonic polynomial on $B(0,1)$ of degree $r$.  Let $\Lc(q; \pv)$ be the lattice in $\R^m$ given by 
$$\Lc(q;\pv)=\{\av=(a_1,\dots,a_m)\in \Z^m\,:\, \av\cdot\pv\equiv 0\mbox{\,\,mod\,\,}q\}$$
where $\av\cdot\pv$ denotes the Euclidean inner product of the vectors $\av$ and $\pv$ in $\R^m$.
 In \cite{LMR}, it is shown that $r$ as above is the minimum $L^1$ norm of nonzero elements of $\Lc(q;\pv)$.  Thus we have:
$$\sigma_2(\orb(q;\pv))=\min\{|a_1|+\dots +|a_m|: 0\ne\, \av\in\Lc(q;\pv)\}.$$

Consider the sequence of Riemannian orbifolds $\orb(q_j;\pv_j)$, $j=1,2,\dots$ where $q_j=j^m$ and $\pv_j=(1, j, j^2,\dots, j^{m-1} )$.
It is easy to see that the minimum $L^1$ norm of vectors in $\Lc(q_j; \pv_j)$ is precisely $j$, attained by the vector $\av=(0,\dots, 0, j)$.  Hence we have
$$\sigma_2(\orb(q_j;\pv_j))=j=(q_j)^{\frac{1}{m}}$$
for all $j$, giving us sharpness of the upper bound in Equation~\eqref{sig2} for $k=2$.  Since $\sigma_k\geq \sigma_2$ when $k\geq 2$, the power of $q$ appearing in Equation~\eqref{sig2} is minimal for every choice of $k$.  This proves the sharpness statement in Theorems~\ref{h-orbibd} and \ref{h-confinv}.

 \end{ex}

In some cases, as illustrated by the example above, we may simplify the upper bounds in \eqref{h-orbibd1} and \eqref{h-confinv1} by obtaining estimates on the isoperimetric ratio $\iso_g(\orb)$.  
We first recall that there are a number of classes of bounded subdomains $\Omega$ of Riemannian manifolds $(M,g)$ whose isoperimetric ratios satisfy a uniform lower bound. For example:
\begin{enumerate}
\item It is a classical result that when $(M,g)$ is isometric to the Euclidean space $\R^n$, the round hemisphere,  or the hyperbolic space $\hyp^n$,  the isoperimetric ratio of $\Omega\subset M$ is bounded below by a constant depending only on the dimension of $M$. 
\item When $(M,g)$ is a Hadamard manifold, the same conclusion as in (1) holds for 
$\Omega \subset M$ by a result of C. Croke \cite{Cro84}.

\item The same conclusion as in (1) holds if $\Omega$ is a subset of a  ``sufficiently small" ball of a Riemannian manifold $(M,g)$ (see, e.g., \cite[p. 136]{Cha01}). A ball $B(x,R)\subset M$ of radius $R$ centered at $x$ is called \textit{sufficiently small} if 
\[
R\le \sup\{s>0 \ | \text{ for all } y\in B(x,s),~\inj(y)\ge 2s\},
\]  
where $\inj(y)$ is the injectivity radius of $M$ at point $y$. 
\item When $M$ is compact, one can also bound the isoperimetric ratio of domains with sufficiently small measure by a constant depending only on the dimension (see \cite[Appendix C]{BM82}). More precisely, 
for every $\epsilon>0$ and every compact Riemannian manifold $(M,g)$, there exists a constant $V=V(M,g,\epsilon)$  such that 
for every $\Omega\subset M$ with $\vol_g(\Omega)\le V$, we have 
\[\iso_g(\Omega)\ge(1-\epsilon)\iso(B(0,1)),\]
where $\iso(B(0,1))$ is the isoperimetric ratio of the unit Euclidean ball $B(0,1) \subset\R^n$.  
\end{enumerate}

B. Colbois, A. El Soufi, and A. Girouard \cite[Theorem 3.3, Corollary 3.4]{CEG} used  some of these bounds on $\iso_g(\Omega)$  to estimate bounds on the Steklov eigenvalues in the smooth setting.
These results do not generalize to the orbifold setting.  However, applying Lemma~\ref{lem.q}, we obtain the following result:

 \begin{prop}\label{h-oig} Let $\mathcal{C}$ be a class of compact $n$-dimensional Riemannian manifolds $(\Om,g)$ whose isoperimetric ratios satisfy a uniform lower bound, say $\iso_g(\Omega)>C>0$ for all $(\Omega,g)\in \mathcal{C}$.   (E.g., $\mathcal{C}$ may be any of the classes of subdomains of Riemannian manifolds enumerated above.)  Then there exist positive constants $c_1$ and $c_2$ depending only on $\mathcal{C}$ and $n$ such that the following holds:  For every orbifold $(\orb,g)$ of the form $(\orb,g)=(\Gamma\bs \Omega,g)$ where $(\Omega,g)\in \mathcal{C}$ and $\Gamma$ is a group of finite order $q$ that acts effectively on $\Omega$ by isometries, we have    
  
  \begin{equation}\label{h-growth}\sigma_k(\orb, g)\vol_g(\partial\orb)^{\frac{1}{n-1}}\le{q^{\frac{1}{n}\left(1-\frac{1}{n-1}\right)}}\left( {c_1\conf(\orb,g)+c_2k^{\frac{2}{n}}}\right). \end{equation}

 \end{prop}

\subsection{Topological upper bounds on orbisurfaces}\label{sect:dim2}

In this subsection, we focus on the 2-dimensional case. As noted in Remark~\ref{typeI_II}, the boundary components of a compact orbisurface consist of closed one-dimensional orbifolds either of type I (a circle) or of type II (a quotient of a circle by a reflection).

\begin{lem}\label{h-orbiembed}
   Every compact Riemannian orbisurface $(\orb,g)$ with $r$ type I boundary components and $s$ type II boundary components has an admissible extension $(\OP,h)$ with $\chi(\OP)=\chi(\orb)+r+\frac{s}{2}$. Moreover, $(\orb,g)$ isometrically embeds in $(\OP,h)$.
\end{lem}

\begin{proof}
We will isometrically embed $(\orb, g)$ in a closed Riemannian orbisurface $(\OP,h)$ with $\chi(\OP)=\chi(\orb)+r+\frac{s}{2}$. Since $\OP$ is closed, the Ricci curvature will necessarily be bounded below and thus $(\OP,h)$ will be an admissible extension.

If there are no type II boundary components, then we cap off each boundary component of $\orb$ as was done for smooth surfaces in \cite{CEG}:  Extend $(\orb,g)$ by adding a small collar neighborhood about each boundary component and smoothly extending the metric $g$, then smoothly glue in a disk with a Riemannian metric without altering the metric on $\orb$.  Let $(\OP,h)$ be the resulting closed surface and observe that $\chi(\OP)=\chi(\orb)+r$.

Next, if $s>0$, the presence of boundary components of type II implies that $\orb$ contains reflectors and is thus non-orientable.  By Remark~\ref{covering}, we can double $\orb$ across all reflector edges to obtain a two-fold Riemannian covering $\pi:(\tilde{\orb},\tilde{g})\to (\orb,g)$.  The orbifold $\tilde{\orb}$ has $2r+s$ boundary components, all of type I, and admits a reflection symmetry  $\tau$ with $\pi\circ\tau=\pi$.   As in the case $s=0$, cap off each boundary component of $\tilde{\orb}$ to obtain a closed orbifold $(\tilde{\OP}, \tilde{h})$, doing so in such a way that the symmetry $\tau$ extends to a reflection symmetry of $(\tilde{\OP}, \tilde{h})$.   The desired extension of $(\orb, g)$ is given by $\OP:=\langle\tau\rangle\bs\tilde{\OP}$ with the metric $h$ induced by $\tilde{h}$.  By Remark~\ref{rem.euler}(i), we have
$$\chi(\OP)=\frac{1}{2}\chi(\tilde{\OP})=\frac{1}{2}(\chi(\tilde{\orb})+2r +s)=\chi(\orb)+r+\frac{s}{2}.$$

\end{proof}

\begin{thm}\label{h-2dupbd}
Let $(\orb,g)$ be a compact Riemannian orbisurface with $r$ type I boundary components and $s$ type II boundary components. Then for every $k\in\N$
\begin{equation}
\sigma_k(\orb,g)\ell_g(\partial\orb)\le
\begin{cases}
     Bk, & \text{if}\ \chi(\orb)+r+\frac{s}{2}\geq 0 ,\\
     -A(\chi(\orb)+r+\frac{s}{2})+Bk, & \text{if}\ \chi(\orb)+r+\frac{s}{2}< 0,
\end{cases}
\end{equation}
where $A$ and $B$ are positive universal constants.
\end{thm}

\begin{proof}
By Theorem~\ref{h-confinv}, we  get that for every $k\in\N$
\[
\sigma_k(\orb,g)\ell_g(\partial\orb)\leq C_1\conf(\orb,g)+C_2k.
\]
We now want to bound $\conf(\orb,g)$ from above.  As in the proof of Lemma~\ref{h-orbiembed}, we isometrically embed $(\orb,g)$ in a closed Riemannian orbisurface $(\OP,h)$ as a subdomain, with $\chi(\OP)=\chi(\orb)+r+\frac{s}{2}$. We divide the proof into cases according to the sign of $\chi(\OP)$.

If $\chi(\OP)<0$, then $\OP$ is a good orbisurface and admits a metric $h_0$ of constant curvature $-1$ conformally equivalent to $h$ by Proposition~\ref{uniform}. We have $$\conf(\orb,g)\leq \vol_{h_0}(\orb)\leq \vol_{h_0}(\OP)=-2\pi\chi(\OP)$$ where the equality follows from the Gauss-Bonnet Theorem (Theorem~\ref{thm:gauss_bonnet}).

If $\chi(\OP) \geq 0$, then $\OP$ admits a metric of nonnegative curvature conformally equivalent to $h$ by Proposition~\ref{uniform}. Thus $\conf(\orb,g)=0$ by Remark~\ref{epsilontrick} and the proof is complete.
\end{proof}

\subsection{Summary of the proof of Theorem~\ref{h-orbibd}}\label{sect:h-orbibd_pf}
The proof follows the same lines as that of \cite[Thm.~4.1]{H}.  To bound the $k^{th}$ Steklov eigenvalue using the minimax characterization as in Proposition~\ref{prop.var}, we will construct $k$ test functions $f_1, \dots, f_k$ with disjoint support and observe that
\[
\sigma_k(\orb,g)\leq\max_j \frac{\int_{\orb} |\nabla_g f_j|^2 d\vol_{(\orb,g)}}{\int_{\partial\orb} f_j^2 d\vol_{(\partial\orb,g)}}.
\]

\noindent \textbf{Step 1.} We first introduce a family of disjoint domains on $\orb$ which shall be used as the supports of the test functions. A metric space is said to satisfy the $(2,N,\rho)-$covering property if each ball of radius $0<r\leq\rho$ can be covered by $N$ balls of radius $r/2$.

\begin{lem}\label{PCovering}
 Let 
$(\OP,{h_0})$ be an $n$-dimensional Riemannian orbifold with $\ric_{(\OP,{h_0})}\ge-\alpha(n-1)$, $\alpha\ge0$. Then  the metric space $(\OP,d_{h_0})$, where $d_{h_0}$ is the Riemannian distance on $\OP$, satisfies the $(2,N,\frac{1}{\sqrt{\alpha}})-$covering property, where $N$ depends only on the dimension of $\OP$ and where $\frac{1}{\sqrt{\alpha}}$ is understood to be infinity if $\alpha=0$.
\end{lem}
\begin{proof}

One can show that the minimal  number $m$ of balls of radius $r/2$ needed to cover a ball $B(x,r)$ of radius $r$ in  $(\OP,h_0)$ is bounded by
\[
m\leq\sup_{p\in B(x,r)}\frac{\vol_{(\OP,h_0)}(B(p,4r))}{\vol_{(\OP,h_0)}(B(p,\frac{r}{4}))}.
\]
(See, for example, the beginning of the proof of Proposition~3.1 in \cite{H}, through the second displayed formula.) Applying the Relative Volume Comparison Theorem for orbifolds (Proposition~\ref{b-g}) and letting  $v(n,0,r)$ denote the volume of an $n$-dimensional ball of radius $r$ in the simply connected space form of curvature zero, we have 
$$m\leq
\left\{\begin{array}{ll}
\frac{v(n,0,4r)}{v(n,0,r/4)}=2^{4n}& \text{if }\alpha=0,\\
\frac{\int_0^{4r}\sinh^{n-1}\sqrt{\alpha}t\ dt}{\int_0^{r/4}\sinh^{n-1}\sqrt{\alpha}t\ dt}\leq2^{4n}e^{(n-1)\sqrt{\alpha}4r}&\text{if }\alpha>0.
\end{array}\right.$$
Thus a ball of radius $r$, with $0<r\le\frac{1}{\sqrt{\alpha}}$, can be covered by $N=\lfloor2^{4n}e^{4(n-1)}\rfloor$  balls of radius $r/2$.   
\end{proof}

Let $(\OP,h_0)$ be an orbifold satisfying the hypotheses of Lemma~\ref{PCovering}. The fact that the locally compact, complete metric space $(\OP,d_{h_0})$ satisfies the $(2,N,\rho)-$covering property with $\rho=\frac{1}{\sqrt{\alpha}}$ allows us to apply \cite[Thm.~2.1]{H} (see also \cite{HTh} for general $\rho$) to obtain:

\begin{prop}\label{capacitors} Let $\nu$ be any finite, non-atomic Borel measure on $\OP$.
There exists a constant $c_1$ depending only on $n$ such that for each $k\in\N$, there exist two families of open sets $\{F_i\}_{i=1}^k$, $\{G_i\}_{i=1}^k$, $\bar{F_i}\subset G_i$ such that
\begin{itemize}
     \item[(i)] $G_i$ are mutually disjoint,
     \item[(ii)] $\nu_g(F_i)\geq\frac{\nu_g(\OP)}{c_1k}$; 
     \item[(iii)] Either $(a)$ all $F_i$ are annuli, i.e., $F_i=A(x_i,r_i,R_i)$ for some $x_i\in\OP$ and $0\leq r_i<R_i$, and $G_i=2F_i=A(x_i,\frac{r_i}{2},2R_i)$ with $2R_i<\frac{1}{\sqrt{\alpha}}$, or $(b)$ all $F_i$ are domains in $\OP$ and $G_i=\{x\in\OP:d_{h_0}(x,F_i)< r_0:=\frac{1}{1600\sqrt{\alpha}}\}$.
\end{itemize}
\end{prop}

We apply the proposition with 
$$\nu_g(W):=\vol_{(\partial\orb,g)}(W\cap\partial\orb).$$

\noindent \textbf{Step 2.} We define our test functions to have support in the open sets $G_i$ defined in Proposition~\ref{capacitors}. If the families $\{F_i\}_{i=1}^k$, $\{G_i\}_{i=1}^k$ satisfy case $(iii)(a)$ of Proposition~\ref{capacitors}, we set
\[f_j(x)=\left\{
\begin{array}{clll}
    1 & {\rm if } & x\in F_j \\
    \frac{2d_{h_0}(x,B(x_j,r_j/2))}{r_j} & {\rm if } & x\in A(x_j,r_j/2,r_j)=B(x_j,r_j)\setminus B(x_j,r_j/2)\\
    1-\frac{d_{h_0}(x_j,B(x,R_j))}{R_j} & {\rm if } & x\in A(x_j,R_j,2R_j)=B(x_j,2R_j)\setminus B(x_j,R_j)\\
    0  & {\rm if } & x\in \OP\setminus G_j\
\end{array}\right..\]
If they satisfy case $(iii)(b)$, we set
\[f_j(x)=\left\{
\begin{array}{cll}
    1 & {\rm if } & x\in F_j \\
    1-\frac{d_{h_0}(x,F_j)}{r_0} & {\rm if } & x\in \left(G_j\setminus F_j\right) \\
    0 & {\rm if } & x\in \OP\setminus G_j  \
\end{array}\right..\]

The computation of the bound on $\sigma_k(\orb,g)$ is then identical to that found in the proof of \cite[Thm.~4.1]{H}\footnote{The proof of \cite[Thm.~4.1]{H} uses the fact that the $f_i$ are Lipschitz functions and thus $\nabla_{h_0}f_i$ exist off a set of measure zero and have bounded norm. In the orbifold case, the $f_i$ are Lipschitz and the set of measure zero includes the orbifold singular set.}.

\subsection{Upper bounds for Neumann Laplace eigenvalues on orbifolds}\label{laplace}
The method used to prove Theorem~\ref{h-orbibd} was also used to obtain upper bounds for the eigenvalues of the Laplace--Beltrami operator on Riemannian manifolds in \cite{H}. We remark that Theorem~1.2 of \cite{H} for the Laplace spectrum remains true in the orbifold setting. More precisely, under the assumptions of Theorem~\ref{h-orbibd}, for every $k\in \N$, the $k^{\rm th}$ Neumann Laplace eigenvalue $\lambda_k(\orb,g)$ satisfies
\begin{equation*}
\lambda_k(\orb,g)\vol_g(\orb)^{\frac{2}{n}}\le C_1\alpha\vol_{h_0}(\orb)^{\frac{2}{n}}+C_2k^{\frac{2}{n}},
\end{equation*}
where $C_1$ and $C_2$ are positive constants depending only on $n$.
Moreover, let $(\orb,g)$ be a compact $n$-dimensional Riemannian orbifold with Neumann boundary condition. Then, in the notation of Theorem~\ref{h-confinv}, for every $k\in \N$, the $k^{\rm th}$ Neumann Laplace eigenvalue $\lambda_k(\orb,g)$ satisfies
\begin{equation*}\lambda_k(\orb,g)\vol_g(\orb)^{\frac{2}{n}}\le C_1\conf(\orb,g)+C_2k^{\frac{2}{n}},
\end{equation*}
where again $C_1$ and $C_2$ are positive constants depending only on $n$.
In the 2-dimensional case, using the notation of Theorem~\ref{h-2dupbd}, we have
\begin{equation*}
\lambda_k(\orb,g)\vol_g(\orb)\le
\begin{cases}
     Bk, & \text{if}\ \chi(\orb)+r+\frac{s}{2}\geq 0 \\
     -A{(\chi(\orb)+r+\frac{s}{2})}+Bk, & \text{if}\ {\chi(\orb)+r+\frac{s}{2}}< 0
,\end{cases}
\end{equation*}
where $A$ and $B$ are positive universal constants.

\bibliographystyle{plain}
\bibliography{orbifoldbib}

\end{document}